\documentclass[10pt,a4paper]{elsarticle}
\usepackage{amsmath}
\usepackage{amsfonts}
\usepackage{amssymb}
\usepackage{graphicx}
\usepackage{caption}
\usepackage{subfig}
\usepackage{tabularx}
\usepackage{color}
\usepackage{algorithmicx} 
\usepackage{algorithm}
\usepackage{amsthm}
\usepackage{float}
\usepackage[toc,page]{appendix}
\newtheorem{theorem}{Theorem}[section]
\newtheorem{lemma}[theorem]{Lemma}
\newtheorem{proposition}[theorem]{Proposition}
\newtheorem{corollary}[theorem]{Corollary}

\journal{Computers and Fluids}

\begin{document}

\begin{frontmatter}


\title{A simplified Cauchy-Kowalewskaya procedure for the implicit solution of generalized Riemann problems of hyperbolic balance laws}

\author[mymainaddress]{Gino I. Montecinos\corref{mycorrespondingauthor}}
\cortext[mycorrespondingauthor]{Corresponding author}
\ead{gino.montecinos@uaysen.cl}
\author[mysecondaryaddress]{Dinshaw S. Balsara}

\address[mymainaddress]{Department of Natural Sciences and Technology, Universidad de Ayse\'en, Coyhaique, Chile}
\address[mysecondaryaddress]{Department of Physics, University of Notre Dame, USA}

\begin{abstract}

The Cauchy-Kowalewskaya (CK) procedure is a key building block in the design of solvers for the Generalised Rieman Problem (GRP) based on Taylor series expansions in time.  The CK procedure allows us to express time derivatives in terms of purely space derivatives. This is a very cumbersome procedure, which often requires the use of software manipulators.
In this paper, a simplification of the CK procedure is proposed in the context of implicit Taylor series expansion for GRP, for hyperbolic balance laws in the framework of [Journal of Computational Physics 303 (2015) 146-172]. A recursive formula for the CK procedure, which is straightforwardly implemented in computational codes, is obtained.
The proposed GRP solver is used in the context of the ADER approach and several one-dimensional problems are solved to demonstrate the applicability and efficiency of the present scheme.  An enhancement in terms of efficiency, is obtained. Furthermore, the expected theoretical orders of accuracy are achieved, conciliating accuracy and stability. 

\end{abstract}


\end{frontmatter}

\section{Introduction}

This paper concerns the solution of Generalized Riemann Problems (GRP) in the context of high-order finite volume methods. The ADER (Arbitrary Accuracy DERivative Riemann problem method), first put forward by Toro et al. \cite{Millington:2001a}, is of particular interest in this work. The method in \cite{Millington:2001a} was devoted to develop a procedure able to compute the numerical solution, of the one-dimensional linear advection problem, of arbitrary order of accuracy in both space and time. This method can be considered as a generalization of Godunov's method, where the numerical fluxes can be obtained from the local solution of GRP where the initial condition
consists of polynomial functions of suitable order. Subsequently, ADER was extended to solve linear systems of hyperbolic conservation laws in \cite{Schwartzkopff:2002a, Toro:2001c}. The ADER philosophy was extended by Toro and Titarev in \cite{Toro:2002a} to solve the non-linear systems; inhomogeneous Burgers equation and the nonlinear shallow-water equations with variable bed elevation. In \cite{Titarev:2002a}, the ADER approach was extended, by Toro and Titarev, to nonlinear but homogeneous hyperbolic systems. Furthermore, the extension of ADER to scalar balance laws was investigated by Toro and Takakura, \cite{Takakura:2003a}.

The original ADER scheme \cite{Millington:2001a, Titarev:2002a}, has two main steps, reconstruction and flux calculation. The marching in time generates cell averages, then the reconstruction procedure generates a special type of interpolation polynomial of the solution from cell averages.  The flux calculation is carried out from the solution of the GRP, which is proposed in terms of a Taylor series expansion in time, where the time derivatives are completely expressed in terms of spatial derivatives by means of the Cauchy-Kowalewskaya or Lax-Wendroff procedure. The space derivatives are obtained from homogeneous linearized Riemann problems constructed from the governing equation and an initial condition given by the derivatives of interpolation polynomials. This was the summary of the pioneering ADER method, which in principle is able to generate approximations of arbitrary order of accuracy. 

The accuracy of ADER methods depend on the number of terms in the Taylor series expansions.  In the particular case of the first order, it recovers the Godunov method and for the second order, ADER recovers the second-order GRP method of Ben-Artzi and Falcoviz \cite{BenArtzi:1984a}.  In \cite{Castro:2008a}, a re-interpretation carried out by Castro and Toro of the high-order numerical method proposed by Harten et al. \cite{Harten:1987a} has allowed us to formulate GRP solutions and thus ADER schemes in a different way. In this new interpretation of ADER, a classical Riemann problem is built from the governing equation and a piece-wise initial condition which is formed from two constant states. In this approach, the constant states are local predictors of the solution within computational cells, that is, these are extrapolation values of the solution at both sides of the cell interfaces at a given time. The evolution of these extrapolated values is carried out by using Taylor series expansion in time, where the time derivatives are still expressed via Cauchy-Kowalewskaya functionals but filled with the spatial derivates of the reconstruction interpolation functions. In this form the predictors in two adjacent computational cells are interacted at the cell interface through the classical Riemann problem. In \cite{Castro:2008a}, the approach based on the Harten et al. is called the HEOC solver and the original GRP solver of Toro and Titarev scheme is referred to as the TT solver.  So, the difference between the HEOC formulation and the TT formulation is that, in the HEOC case only one classical Riemann problem is required but it needs to be solved at each quadrature point, whereas, in the TT approach a sequence of classical Riemann problems are required, only once, one for the leading term of Taylor expansions and linearised Riemann problems for the spatial derivatives. So, after these terms are available the computation of GRP solution via TT is reduced jut to evaluate a polynomial in time. The similarities between both approaches are the use of Taylor series expansion and the use of the Cauchy-Kowalewskaya procedure. A detailed review of GRP solvers is done in \cite{Castro:2008a, Montecinos:2012a}. Similarly, further details of ADER schemes can be found in Chapters 19 and 20 of the textbook by Toro \cite{Toro:2009a}. Notice that, the GRP solvers require the ability of solving classical Riemann problems. The number of hyperbolic systems where the exact solution of Riemann problems is available, is limited. In general, the exact solution of Riemann problems for several hyperbolic system can be very difficult to be obtained. Fortunately, the ADER approaches described above can use approximate Riemann solvers, see \cite{Goetz:2012a}.  In this sense the search for approximate Riemann solvers for general hyperbolic system is very relevant area of research, where the ADER philosophy can benefit. Balsara \cite{Balsara:2010a,Balsara:2012b} and Balsara et al. \cite{Balsara:2014a} have extended multidimensional HLL and HLLC to Euler and MHD equations.   Goetz et al. \cite{Goetz:2018a} have shown that approximate Riemann solver can be obtained from the well-known HLL solver. See also \cite{Dumbser:2016a, Balsara:2018b} where universal approximate GRP solvers based on HLL method and the inclusion of intermediate waves, have been reported.

The ADER approach allows flexibility to incorporate the finite element approach into the finite volume framework. In an intermediate stage in the mixing between finite volume and finite element approach, Dumbser and Munz, \cite{Dumbser:2005a, Dumbser:2006a}, have implemented the ADER approach for the Discontinuous Galerkin approach applied to the aeroacoustics and the Euler equations in two dimensions. In this approach, the ADER is used to evolve polynomials in space and time, through the evolution of their degrees of freedom by using Taylor series expansion and the Cauchy-Kowalewskaya procedure but instead of using reconstruction and the derivatives of reconstruction polynomials, the authors proposed to use the test functions of the finite element space. In this sense, the work of Dumbser et al. \cite{Dumbser:2008a} in their pioneering work, has introduced the Galerkin framework for obtaining the predictor within cells. The difference between this approach and that of Dumbser and Munz is that this neither requires the use of any Taylor series expansion nor Cauchy-Kowalewskaya procedure. Some contributions to the developments of this class of solver can be found in \cite{Dumbser:2008b, Dumbser:2009b, Dumbser:2009c, Goetz:2016a, Montecinos:2017c}, to mention but a few.

The methodologies based on Galerkin approaches require the inversion of matrices and the solution of non-linear algebraic equations which are very time-consuming processes. On the other hand, the methodologies based on Taylor series suffer from the Cauchy-Kowaleskaya procedure which becomes cumbersome when the accuracy increases. Furthermore, the Taylor series expansions without modifications cannot deal with stiff source terms.  In this sense, in \cite{Montecinos:2014c, Toro:2015a} Montecinos and Toro have introduced the implicit Taylor series expansion and Cauchy-Kowakeskaya procedure to deal with hyperbolic balance laws with stiff source terms. In this approach, the Cauchy-Kowakeskaya procedure requires the spatial derivatives evolved in time, which cannot be obtained straightforwardly from Riemann problems as in conventional ADER methods discussed above. This is basically because conventional ADER methods based on TT solver uses linearised Riemann problems which are homogeneous. Thus, the influence of the source term is only involved in the Cauchy-Kowalewskaya procedure. In \cite{Toro:2015a} the spatial derivatives are evolved by using two approaches, which differ in the number of terms in the Taylor expansion, so Complete Implicit Taylor Approach (CITA) and Reduced Implicit Taylor Approach (RITA), in both HEOC and TT approaches, are investigated there.  A limitation of the approach in \cite{Toro:2015a}, is the computational cost, for high orders of accuracy, the CPU time increases dramatically. However, second order approaches work very well as reported in \cite{Vanzo:2016a} where an extension to transport phenomena on unstructured meshes has been reported. 

In this paper, we propose a strategy related to the implicit Taylor series expansion and the Cauchy-Kowalewskaya procedure. The difference between the conentional approaches described above and the proposed one is that spatial derivatives do not need to be evolved and the Cauchy-Kowalewskaya procedure is modified by expressing high-order time derivatives not only in terms of spatial derivatives of the data but also on space and time derivatives of the Jacobian matrices of the flux and source functions as well. The derivatives are obtained from an interpolation fashion on selected nodal points. This simplification allows us to provide a closed form for the Cauchy-Kowalewskaya functionals, in a recursive formula. Furthermore, this approach requires the solution of one algebraic equation but the number of variables is the same for all orders of accuracy. Closed forms of the Cauchy-Kowaleskaya functional are available for some partial differential equations as linear advection systems, linear systems with constant matrices \cite{Schwartzkopff:2004a, Kaser:2006a} and for the non-linear two-dimensional Euler equation, \cite{Dumbser:2006a}. However, the expression obtained here is useful for all hyperbolic balance laws. 

This paper is organized as follows.  In section \ref{section:Framework}, the general framework is presented. In section \ref{section:Predictor}, the new predictor step is introduced. In the section \ref{section:Results}, numerical tests are concerned. Finally,  in section \ref{section:Conclusions} conclusions and remarks are drawn.

\section{The framework}\label{section:Framework}
In this paper we present a strategy for solving a hyperbolic balance law in the conservative form
\begin{eqnarray}
\label{eq:gov-eq-1}
\begin{array}{c}
\partial_t \mathbf{Q} + \partial_x \mathbf{F} ( \mathbf{Q}) = \mathbf{S} ( \mathbf{Q}) \;, \\

\mathbf{Q}(x,0) = \mathbf{H}_0(x)
\;,
\end{array}
\end{eqnarray}
where $ \mathbf{H}_0(x)$ is a prescribed function in $\mathbb{R}^m$. Here $\mathbf{Q}(x,t)\in \mathbb{R}^{m}$ is the vector of unknowns,   $\mathbf{F}(\mathbf{Q})\in \mathbb{R}^{m}$ is the physical flux function and $\mathbf{S}(\mathbf{Q})\in \mathbb{R}^{m}$ is the source term.

To compute a numerical solution of (\ref{eq:gov-eq-1}), we divide computational domains into $N$ uniform cells of the form $ I_i^n := [x_{i-\frac{1}{2}}, x_{i+\frac{1}{2}}]\times [t^{n}, t^{n+1}]$ and then by integrating on $I_i^n$ we obtain the well-known one-step formula.
\begin{eqnarray}
\begin{array}{c}
\mathbf{Q}_i^{n+1} = \mathbf{Q}_i^{n} - \frac{\Delta t}{ \Delta x} \biggl(  \mathbf{F}_{i+\frac{1}{2}} - \mathbf{F}_{i-\frac{1}{2}} \biggr)  + \Delta t \mathbf{S}_i\;,
\end{array}
\end{eqnarray}
where 
\begin{eqnarray}
\begin{array}{c}
\mathbf{Q}_i^{n} = \frac{1}{\Delta x} \displaystyle \int_{x-\frac{1}{2} }^{ x + \frac{1}{2} } 
\mathbf{Q}(x,t^{n} ) dx 
\end{array}
\end{eqnarray}
and the numerical flux $\mathbf{F}_{i+\frac{1}{2} } $ as well as the source term $\mathbf{S}_{i } $ are computed by adopting the ADER strategy, \cite{Millington:2001a, Toro:2001g, Toro:2002a,Toro:2009a}. In this paper
\begin{eqnarray}
\label{eq:flux-source}
\begin{array}{c}

\mathbf{F}_{i+\frac{1}{2} } = 
\frac{1}{ \Delta t} \displaystyle 
\int_{ t^{n} }^{ t^{n+1}}  
\mathbf{F}_h( \mathbf{Q}_{i}(x_{ i+\frac{1}{2} },t), \mathbf{Q}_{i+1} (x_{ i+\frac{1}{2} },t) ) dt\;,

\\

\mathbf{S}_{i} = \frac{1}{\Delta t \Delta x } \displaystyle \int_{t^{n}}^{t^{n+1}} \int_{x-\frac{1}{2}}^{x+\frac{1}{2}}
 \mathbf{S}(\mathbf{Q}_i(x,t) dx dt\;,
\end{array}
\end{eqnarray}
where $\mathbf{F}_h(\mathbf{Q}_L, \mathbf{Q}_R)$  is a classical numerical flux function, which  can be seen as a function of two states $  \mathbf{Q}_L$ and  $ \mathbf{Q}_R$. This can be an approximate resolved flux in a Riemann problem. It is possible to design an approximate GRP solver out of the HLL, \cite{Goetz:2018a}. In this paper we will use the Rusanov solver, obtained from HLL by taking extreme left and right maximum waves speed to be the same but in opposite directions. Here $\mathbf{Q}_i(x,t)$ corresponds to a predictor within the computational cell $I_i^n$.  In the next section, further details of the predictor step are provided.

\section{The predictor step}\label{section:Predictor}

In this section we provide the details to obtain the predictor $\mathbf{Q}_{i}(x,t)$. We adopt the strategy of the implicit Taylor series expansion and the Cauchy-Kowalewskaya procedure presented in \cite{Toro:2015a}. However, instead of the conventional Cauchy-Kowalewskaya procedure we use a simplified version of this procedure. For the sake of completeness, we provide a brief review of the approach in \cite{Toro:2015a}. The predictor is computed as 
\begin{eqnarray}
\begin{array}{c}
\mathbf{Q}_i(x,\tau ) = \mathbf{Q}_i( x, 0_+) - \sum_{k=1}^{M} \frac{(-\tau)^k}{k!} \partial_t^{(k)} \mathbf{Q}_i(x, \tau) \;,
\end{array}
\end{eqnarray}
which by means of the Cauchy-Kowalewskaya procedure, can be written as 
\begin{eqnarray}
\begin{array}{c}
\mathbf{Q}_i(x,\tau ) = \mathbf{Q}_i( x, 0_+) - \sum_{k=1}^{M} \frac{(-\tau)^k}{k!} \mathbf{G}^{(k)}( \mathbf{Q}_i(x, \tau), \partial_x \mathbf{Q}_i(x, \tau), ...,\partial_x^{(k)}\mathbf{Q}_i(x, \tau)  )\;,
\end{array}
\end{eqnarray}
where $M$ is an integer which corresponds to the order of accuracy  $M+1$, in both space and time.  Here, $\mathbf{G}^{(k)}$ corresponds to the Cauchy-Kowalewskaya functional and it is the function which expresses the time derivatives in terms of spatial derivatives, $\partial_t^{(k)} \mathbf{Q}_i(x, \tau) = \mathbf{G}^{(k)}( \mathbf{Q}_i(x, \tau), \partial_x \mathbf{Q}_i(x, \tau), ...,\partial_x^{(k)}\mathbf{Q}_i(x, \tau) )$. Notice that this functional requires the information of spatial derivatives at $\tau$, which must be estimated. In \cite{Toro:2015a} two strategies are proposed, where the time spatial derivatives are obtained from implicit Taylor series as well. These approaches require the solution of algebraic equations.

In the next section we provide a brief review of the conventional Cauchy-Kowalewskaya procedure and subsequently, a simplification of this procedure is presented.

\subsection{A brief review of the Cauchy-Kowalewskaya procedure}\label{sec:review-CK}
Here, we briefly describe the Cauchy-Kowalewskaya procedure for obtaining the time derivatives of the data. For the sake of simplicity, in this section we omit the sub index $i$ in $\mathbf{Q}_i$, to indicate the approximation within the cell $[x_{i-\frac{1}{2}},  x_{i+\frac{1}{2}}].$ 

The time derivatives are obtained from the governing equation (\ref{eq:gov-eq-1}). As for example, the first derivative is given by
\begin{eqnarray}
\label{eq:dtQ}
\begin{array}{c}
\partial_t \mathbf{Q} = - \mathbf{A}(\mathbf{Q}) \partial_x \mathbf{Q} + \mathbf{S}( \mathbf{Q})\;,
\end{array}
\end{eqnarray}
where $\mathbf{A}(\mathbf{Q})$ is the Jacobian matrix of $\mathbf{F}(\mathbf{Q})$ with respect to $\mathbf{Q}$.  So, to obtain the second time derivative we differentiate in time the equation (\ref{eq:dtQ}), so we have
\begin{eqnarray}
\label{eq:dttQ}
\begin{array}{c}
\displaystyle 
\partial_t^{(2)} \mathbf{Q}_i = - \sum_{j=1}^{m} \sum_{l=1}^{m}
\frac{\partial \mathbf{A}_{i,j}(\mathbf{Q}) }{\partial \mathbf{Q}_l} \partial_t \mathbf{Q}_l  \partial_x \mathbf{Q}_j 
- \sum_{j=1}^{m}  \mathbf{A}_{i,j}(\mathbf{Q})  \partial_t ( \partial_x \mathbf{Q}_j ) \\

\displaystyle 
+ \sum_{j=1}^{m} \mathbf{B}( \mathbf{Q})_{i,j} \partial_t \mathbf{Q}_j \;,
\end{array}
\end{eqnarray}
where $\mathbf{B}(\mathbf{Q})$ is the Jacobian matrix of $\mathbf{S}( \mathbf{Q} )$ with respect to $\mathbf{Q}$,  $\partial_t \mathbf{Q}_i$  and  $\partial_t \mathbf{Q}_j$ are the $i$th and the $j$th components of the vector state $\partial_t \mathbf{Q}$, respectively and then by differentiating  the expression (\ref{eq:dtQ}) with respect to $x$, we obtain 
\begin{eqnarray}
\begin{array}{c}
\displaystyle 
\partial_x ( \partial_t \mathbf{Q}_j )
= -\sum_{j=1}^{m} \sum_{l=1}^{m}
\frac{\partial \mathbf{A}_{i,j}(\mathbf{Q}) }{\partial \mathbf{Q}_l} \partial_x \mathbf{Q}_l  \partial_x \mathbf{Q}_j 
-\sum_{j=1}^{m}  \mathbf{A}_{i,j}(\mathbf{Q}) \partial_x^{(2)} \mathbf{Q}_j \\
\displaystyle 
+ \sum_{j=1}^{m} \mathbf{B}( \mathbf{Q})_{i,j} \partial_x \mathbf{Q}_j \;,
\end{array}
\end{eqnarray}
the same procedure is applied to obtain $\partial_t^{(3)} \mathbf{Q}$ and so on, in principle any high order time derivative can be obtained through this procedure.  However, the procedure becomes very cumbersome for derivatives of orders higher than two, furthermore the complexity scales with the number of unknowns $m$ and the order of accuracy as well. This justifies the requirement of finding some efficient strategy to approximate temporal derivatives by following the Cauchy-Kowalewskaya ideas.

\subsection{The simplified Cauchy-kowalewskaya procedure}\label{sec:simplified-CK}

In this section, we derive a simplified Cauchy-kowalewskaya procedure, to approximate time derivatives. 
For the sake of simplicity, in this section we also omit the sub index $i$ in $\mathbf{Q}_i$, to indicate the approximation within the cell $[x_{i-\frac{1}{2}},  x_{i+\frac{1}{2}}].$  As seen in the previous section, the conventional Cauchy-Kowalewsky procedure provides the first time derivative as
\begin{eqnarray}
\begin{array}{c}

\partial_t \mathbf{Q} = - \mathbf{A}(\mathbf{Q})\partial_x \mathbf{Q} + \mathbf{S}(\mathbf{Q}) \;.

\end{array}
\end{eqnarray}
At this point we introduce the {\it first simplification}. Instead of considering the previous equation, we are going to use the approximation

\begin{eqnarray}
\label{eq:simplify-1}
\begin{array}{c}

\partial_t \mathbf{Q} = - \mathbf{A}(x,t)\partial_x \mathbf{Q} + \mathbf{S}(\mathbf{Q}) \;,

\end{array}
\end{eqnarray}
which means, the matrix $\mathbf{A}$ is considered just a space-time dependent matrix.

By taking into account the approximation (\ref{eq:simplify-1}), the second time derivative can be approximated in three steps described below. 

{\bf Step I}. We differentiate $\partial_t \mathbf{Q} $ with respect to $t$, thus
\begin{eqnarray}
\begin{array}{c}

\partial_t ( \partial_t \mathbf{Q}) = - \mathbf{A}_t \partial_x \mathbf{Q} - \mathbf{A} \partial_t ( \partial_x \mathbf{Q})+ \mathbf{ B} \partial_t \mathbf{Q}  \;,

\end{array}
\end{eqnarray}
where $ \mathbf{B}$ is the Jacobian matrix of the source term with respect to $\mathbf{Q}$. For the remaining part of this paper, we use the notation $\partial_x \mathbf{A} = \mathbf{A}_x$ for any matrix $\mathbf{A}$. Similarly, we use the convention $  \partial_x^{(l)} \mathbf{A} = \mathbf{A}_x^{(l)}$ for the $l$-th partial derivative of the matrix $\mathbf{A}$ with respect to $x$. Do not confuse with $\mathbf{A}^{l}$ which means matrix multiplication, $l$ times. 

{\bf Step II}. At this point we introduce the {\it second simplification}, which is to consider also the matrix $\mathbf{B}$ as a space-time dependent matrix rather than a state dependent matrix. We assume regularity enough such that the spatial and time derivatives can be interchanged. So  
\begin{eqnarray}
\label{eq:deriv-eq-ord2}
\begin{array}{c}

\partial_t^{(2)} \mathbf{Q} =
 - \mathbf{A}_t \partial_x \mathbf{Q} - \mathbf{A} \partial_x ( \partial_t \mathbf{Q} ) + \mathbf{ B} \partial_t \mathbf{Q}  \;.

\end{array}
\end{eqnarray}

{\bf Step III}. Here, we differentiate in space the expression $\partial_t \mathbf{Q}$, taking into account the simplifications introduced above, to obtain
\begin{eqnarray}
\begin{array}{c}

\partial_x ( \partial_t \mathbf{Q}) = - \mathbf{A}_x \partial_x \mathbf{Q} - \mathbf{A} \partial_x^{(2)} \mathbf{Q}+ \mathbf{ B} \partial_x \mathbf{Q}  
\\
= - \mathbf{A} \partial_x^{(2)} \mathbf{Q} + (\mathbf{B} - \mathbf{A}_x^{(1)}) \partial_x \mathbf{Q} \;.

\end{array}
\end{eqnarray}
This step regards the main difference with respect to the conventional Cauchy-Kowalewskaya procedure. By inserting the previous expression into (\ref{eq:deriv-eq-ord2}), we obtain 
\begin{eqnarray}
\begin{array}{c}

\partial_t^{(2)} \mathbf{Q} =
 - \mathbf{A}_t \partial_x \mathbf{Q} - \mathbf{A} ( - \mathbf{A} \partial_x^{(2)} \mathbf{Q} + (\mathbf{B} - \mathbf{A}_x^{(1)}) \partial_x \mathbf{Q}   ) + \mathbf{ B} \partial_t \mathbf{Q}  
\\
= \mathbf{A}^{2} \partial_x^{(2)} \mathbf{Q} + ( - \mathbf{A}_t - \mathbf{A} (\mathbf{B} - \mathbf{A}_x^{(1)})  ) \partial_x \mathbf{Q}+ \mathbf{ B} \partial_t \mathbf{Q}  \;.

\end{array}
\end{eqnarray}
The last expression can be written as 
\begin{eqnarray}
\begin{array}{c}

\partial_t^{(2)} \mathbf{Q} = \mathbf{C}(2,2) \partial_x^{(2)} \mathbf{Q} + \mathbf{C}(2,1) \partial_x \mathbf{Q}+ \mathbf{ B} \partial_t \mathbf{Q}  \;,

\end{array}
\end{eqnarray}
where $ \mathbf{C}(2,i)$, $i=1,2$ represent the matrix coefficients
\begin{eqnarray}
\begin{array}{c}

\mathbf{C}(2,2) = \mathbf{A}^2 \;, 
\mathbf{C}(2,1) =   - \mathbf{A}_t - \mathbf{A} (\mathbf{B} - \mathbf{A}_x^{(1)})  \;.

\end{array}
\end{eqnarray}

The novel contribution of this paper is as follows. We are going to show that this procedure can be generalized. Before to give this main result, we need to express $\partial_x^{(l)} (\partial_t \mathbf{Q})$ only in terms of spatial derivatives of the data and the Jacobian matrices. Notice that from simplifications introduced above, we obtain
\begin{eqnarray}
\begin{array}{cl}

\partial_x ( \partial_t \mathbf{Q})
=& - \mathbf{A} \partial_x^{(2)} \mathbf{Q}
+ ( \mathbf{B} - \mathbf{A}_x) \partial_x \mathbf{Q} \;, \\

\partial_x^{ ( 2) } ( \partial_t \mathbf{Q})
=& - \mathbf{A} \partial_x^{(3)} \mathbf{Q} 
+ ( \mathbf{B}   - 2 \mathbf{A}_x)       \partial_x^{(2)} \mathbf{Q}  
+ ( \mathbf{B}_x -   \mathbf{A}_x^{(2)}) \partial_x       \mathbf{Q}  \;, \\

\partial_x^{ ( 3)} ( \partial_t \mathbf{Q})
= & - \mathbf{A} \partial_x^{(4)} \mathbf{Q} 
+ (  \mathbf{B}         - 3 \mathbf{A}_x      ) \partial_x^{(3)} \mathbf{Q}  
+ ( 2\mathbf{B}_x       - 3 \mathbf{A}_x^{(2)}) \partial_x^{(2)} \mathbf{Q} 
\\
& + (  \mathbf{B}_x^{(2)} -   \mathbf{A}_x^{(3)}) \partial_x       \mathbf{Q} 
 \;, \\

\partial_x^{ ( 4) } ( \partial_t \mathbf{Q})
=& - \mathbf{A} \partial_x^{(5)} \mathbf{Q} 
+ ( \mathbf{B}          - 4 \mathbf{A}_x      ) \partial_x^{(4)} \mathbf{Q}  
+ ( 3\mathbf{B}_x       - 6 \mathbf{A}_x^{(2)}) \partial_x^{(3)} \mathbf{Q} \\
& + ( 3\mathbf{B}_x^{(2)} - 4 \mathbf{A}_x^{(3)}) \partial_x^{(2)} \mathbf{Q} 
+ (  \mathbf{B}_x^{(3)} -   \mathbf{A}_x^{(4)}) \partial_x       \mathbf{Q} 
 \;, \\

\partial_x^{ ( 5)} ( \partial_t \mathbf{Q})
=& - \mathbf{A} \partial_x^{(6)} \mathbf{Q} 
+ (  \mathbf{B}         - 5 \mathbf{A}_x      ) \partial_x^{(5)} \mathbf{Q}  
+ ( 4\mathbf{B}_x       -10 \mathbf{A}_x^{(2)}) \partial_x^{(4)} \mathbf{Q} \\
& + ( 6\mathbf{B}_x^{(2)} -10 \mathbf{A}_x^{(3)}) \partial_x^{(3)} \mathbf{Q} 
 + ( 4\mathbf{B}_x^{(3)} - 5 \mathbf{A}_x^{(4)}) \partial_x^{(2)} \mathbf{Q} \\
& + (  \mathbf{B}_x^{(4)} -   \mathbf{A}_x^{(5)}) \partial_x       \mathbf{Q} 
 \;, \\

\end{array}
\end{eqnarray}
so by inspection we observe that these derivatives, we can be arranged  as 
\begin{eqnarray}
\label{eq:time-space-from-only-space}
\begin{array}{c}

\partial_x^{(l)} ( \partial_t \mathbf{Q}) =

\sum_{k=1}^{l+1}(b_{l,k} \mathbf{B}_x^{(l+1-k)} - a_{l, k} \mathbf{A}_x^{(l+2-k)} ) \partial_x^{( k)}\mathbf{Q} \;.

\end{array}
\end{eqnarray}
Notice that,  $\partial_x^{(l)}$ stands by the $l$-th spatial derivative, with the convention $ \partial_x^{(0)} \mathbf{M} = \mathbf{M}$ for any function $\mathbf{M}$, which may be a scalar, vector or matrix function. This notation is also extended to temporal derivatives.

The table \ref{table-coeff-a}, shows the coefficient $a_{l,k}$. Similarly, the table \ref{table-coeff-b} shows the coefficients $b_{l,k}$. We observe that they follow the structure of the Pascal triangle, in the combinatorial theory.  In fact, the structure is given by the following.

\begin{table}
\centering
\begin{tabular}{cccccc|c}
\hline
$ a_{l,(l-4)}$ & $a_{l,(l-3)}$ & $a_{(l,l-2)}$ & $a_{(l, l-1)}$ & $a_{(l,l)}$ & $a_{(l,l+1)}$ & $l$ \\
\hline

0 & 0 &  0 &  0  & 1 &   1 & 1 \\
0 & 0 &  0 &  1  & 2 &   1 & 2 \\
0 & 0 &  1 &  3  & 3 &   1 & 3 \\
0 & 1 &  4 &  6  & 4 &   1 & 4 \\
1 & 5 & 10 & 10  & 5 &   1 & 5 \\
\\
\end{tabular}
\caption{Coefficients $a_{l,k}$ in expression (\ref{eq:time-space-from-only-space}). }\label{table-coeff-a}
\end{table}
\begin{table}
\centering
\begin{tabular}{cccccc|c}
\hline
$ b_{l,(l-4)}$ & $ b_{l,(l-3)}$ & $ b_{(l,l-2)}$ & $ b_{(l, l-1)}$ & $ b_{(l,l)}$ & $ b_{(l,l+1)}$ & $l$ \\
\hline

0 & 0 &  0 &  0  & 0 &   1 & 1 \\
0 & 0 &  0 &  0  & 1 &   1 & 2 \\
0 & 0 &  0 &  1  & 2 &   1 & 3 \\
0 & 0 &  1 &  3  & 3 &   1 & 4 \\
0 & 1 &  4 &  6  & 4 &   1 & 5 \\
\\
\end{tabular}
\caption{Coefficients $b_{l.k}$ in expression (\ref{eq:time-space-from-only-space}). }\label{table-coeff-b}
\end{table}

\begin{lemma}
\begin{eqnarray}
\label{eq:expansion-dxkdt}
\begin{array}{c}

\partial_x^{ ( l)} (\partial_t \mathbf{Q}) = 
\displaystyle 
\sum_{k=1 }^{ l+1} 

\mathbf{D}(l+1, k )

\partial_x^{(k)} \mathbf{Q} \;,
\end{array}
\end{eqnarray}
where 
\begin{eqnarray}
\label{eq:Matrix-D}
\begin{array}{c}

\mathbf{D}(l+1, k ) =  
\biggl(

\left(
\begin{array}{c}
l -1 \\
l - k 
\end{array}
\right)

\mathbf{B}_{x}^{(l-k)}

-

\left(
\begin{array}{c}
l  \\
l + 1 - k 
\end{array}
\right)
\mathbf{A}_{x}^{(l+1-k)}

\biggr)

\end{array}
\end{eqnarray}
and
\begin{eqnarray}
\begin{array}{c}
\left(
\begin{array}{c}
l   \\
- 1
\end{array}
\right) = 0 \;,
\end{array}
\end{eqnarray}
for all integer $l$.
\end{lemma} 
 
\begin{proof}
Let us prove it by induction.

\begin{itemize}
\item We already know that this is true for $ l = 1$. 

\item Let us assume it is true for $l$, that is
\begin{eqnarray}
\label{eq:hip-expansion-dxkdt}
\begin{array}{c}

\partial_x^{ ( l)} (\partial_t \mathbf{Q}) = 
\displaystyle 
\sum_{k=1 }^{ l+1} \biggl(

\left(
\begin{array}{c}
l -1 \\
l - k 
\end{array}
\right)

\mathbf{B}_{x}^{(l-k)}

-

\left(
\begin{array}{c}
l  \\
l + 1 - k 
\end{array}
\right)
\mathbf{A}_{x}^{(l+1-k)}

\biggr)
\partial_x^{(k)} \mathbf{Q} \;.
\end{array}
\end{eqnarray}

\item Let us prove it is true for $l+1$. Indeed

\begin{eqnarray}
\label{eq:expansion-dxkdt-0}
\begin{array}{ll}

\partial_x^{ ( l+1)} (\partial_t \mathbf{Q}) 
&
=   
 
\displaystyle 
\sum_{k=1 }^{ l+1} 

\biggl(

\left(
\begin{array}{c}
l -1 \\
l - k 
\end{array}
\right)

\mathbf{B}_{x}^{(l+1-k)}

-

\left(
\begin{array}{c}
l  \\
l + 1 - k 
\end{array}
\right)
\mathbf{A}_{x}^{(l+2-k)}

\biggr)
\partial_x^{(k)} \mathbf{Q}

\\

&
\displaystyle 
+ 
\sum_{k=1 }^{ l+1} 
\biggl(

\left(
\begin{array}{c}
l - 1 \\
l - k 
\end{array}
\right)
\mathbf{B}_{x}^{(l+1-k)}

-

\left(
\begin{array}{c}
l  \\
l + 1 - k 
\end{array}
\right)
\mathbf{A}_{x}^{(l+1-k)}

\biggr)
\partial_x^{(k+1)} \mathbf{Q} 

\\
&
=

\biggl(

\left(
\begin{array}{c}
l -1 \\
l - 1 
\end{array}
\right)

\mathbf{B}_{x}^{(l+1)}

-

\left(
\begin{array}{c}
l  \\
l  
\end{array}
\right)
\mathbf{A}_{x}^{(l+2)}

\biggr)
\partial_x \mathbf{Q} 

\\
& + 

\displaystyle 
\sum_{k=2 }^{ l+1} 

\bigg[

\left(
\begin{array}{c}
l -1 \\
l + 1 - k 
\end{array}
\right)
+
\left(
\begin{array}{c}
l -1 \\
l - k 
\end{array}
\right)

\biggr]
\mathbf{B}_{x}^{(l+1-k)}

\partial_x^{(k)} \mathbf{Q} 

\\
&
 -
\displaystyle 
\sum_{k=2 }^{ l+1} 
\biggl[
\left(
\begin{array}{c}
l  \\
l + 2 - k 
\end{array}
\right)
+
\left(
\begin{array}{c}
l  \\
l + 1 - k 
\end{array}
\right)

\biggr]
\mathbf{A}_{x}^{(l+2-k)}

\partial_x^{(k)} \mathbf{Q} 

\\
& +

\displaystyle 
\biggl(

\left(
\begin{array}{c}
l -1 \\
- 1 
\end{array}
\right)

\mathbf{B}_{x}^{(-1)}

-

\left(
\begin{array}{c}
l  \\
0  
\end{array}
\right)
\mathbf{A}_{x}^{(0)}

\biggr)
\partial_x^{(l+2)} \mathbf{Q} 
\;.
\end{array}
\end{eqnarray}
By considering the properties of the combinatorial factors
\begin{eqnarray}
\begin{array}{c}

\left(
\begin{array}{c}
l -1 \\
l + 1 - k 
\end{array}
\right)
+
\left(
\begin{array}{c}
l -1 \\
l - k 
\end{array}
\right)

= 

\left(
\begin{array}{c}
l \\
l + 1 - k 
\end{array}
\right)
\;,

\\
\\
\left(
\begin{array}{c}
l-1 \\
l-1 
\end{array}
\right)
=
\left(
\begin{array}{c}
l \\
l 
\end{array}
\right)
=
\left(
\begin{array}{c}
l \\
0 
\end{array}
\right)
=
\left(
\begin{array}{c}
l-1 \\
0 
\end{array}
\right)
= 1
\end{array}
\end{eqnarray}
and the assumption
\begin{eqnarray}
\begin{array}{c}
\left(
\begin{array}{c}
m \\
-1 
\end{array}
\right)
=0 \;,
\end{array}
\end{eqnarray}
for all $m$, after grouping terms we obtain 
\begin{eqnarray}
\label{eq:expansion-dxkdt-1}
\begin{array}{c}

\partial_x^{ ( l+1)} (\partial_t \mathbf{Q}) = 
\displaystyle 
\sum_{k=1 }^{ l+2} \biggl(

\left(
\begin{array}{c}
l \\
l + 1 - k 
\end{array}
\right)

\mathbf{B}_{x}^{(l + 1 - k)}

-

\left(
\begin{array}{c}
l + 1 \\
l + 2 - k 
\end{array}
\right)
\mathbf{A}_{x}^{(l+2-k)}

\biggr)
\partial_x^{(k)} \mathbf{Q} \;.
\end{array}
\end{eqnarray}
This completes the proof.
\end{itemize}

\end{proof}

\begin{proposition}\label{proposition:main-1}

The high-order time derivatives have the following recursive form
\begin{eqnarray}
\label{eq:ReduceCK-formula}
\begin{array}{c}
\displaystyle

\partial_t^{(k)} \mathbf{Q} = \sum_{l=1}^{k} \mathbf{C}(k, l) \partial_x^{(l)} \mathbf{Q} + \partial_t^{(k-2)}( \mathbf{B}\partial_t  \mathbf{Q})\;,

\end{array}
\end{eqnarray}
where 
\begin{eqnarray}
\label{eq:Matrix-C}
\begin{array}{c}
 \mathbf{C}(k, l ) =
 \left\{
 \begin{array}{cc}
 \mathbf{C}(k-1,k-1) \mathbf{D}(k, k), & l = k \;, \\
 \mathbf{C}(k-1,l)_t + \sum_{m=l-1}^{k-1} \mathbf{C}(k-1,m) \mathbf{D}(m+1,l), & l < k \;, \\
 
 \end{array}
 \right. 
\end{array}
\end{eqnarray}
here, the matrix $\mathbf{D}$ is given by (\ref{eq:Matrix-D}). We impose $ \mathbf{C}(k,0) = \mathbf{0}$ $\forall k >0 $, $\mathbf{C}(1,1) = - \mathbf{A}$ and $
\partial_t^{(-1)}( \mathbf{B}\partial_t  \mathbf{Q}) = \mathbf{S}(\mathbf{Q})$.

\end{proposition} 

\begin{proof} Let us prove this proposition by induction. 

\begin{itemize}
\item The result is true for $k = 2$. In fact, we know that
\begin{eqnarray}
\label{eq:hip-1}
\begin{array}{c}
\partial_t \mathbf{Q} = - \mathbf{A} \partial_x \mathbf{Q} + \mathbf{S}(\mathbf{Q}) \;.
\end{array}
\end{eqnarray}
Since, $\partial_t^{(2)} \mathbf{Q} = \partial_t (\partial_t \mathbf{Q})$ from (\ref{eq:hip-1}) we obtain 
\begin{eqnarray}
\label{eq:hip-2}
\begin{array}{c}
\partial_t^{(2)} \mathbf{Q} = \mathbf{A}^{2} \partial_x^{(2)} \mathbf{Q} + (- \mathbf{A}_t  - \mathbf{A}(\mathbf{B} - \mathbf{A}_x) )  \partial_x \mathbf{Q} +  \mathbf{B} \partial_t \mathbf{Q} \;.
\end{array}
\end{eqnarray}
Here, we have used the chain rule 
\begin{eqnarray}
\begin{array}{c}
\partial_t ( \mathbf{S}(\mathbf{Q}) ) = \mathbf{B} \partial_t \mathbf{Q}\;.
\end{array}
\end{eqnarray}
Therefore, from the expressions $ \mathbf{C}(1,1) = - \mathbf{A}$, $\mathbf{D}(2,2)  = -\mathbf{A}$, $\mathbf{D}(2,1) = \mathbf{B} - \mathbf{A}_x$ and by identifying terms, the induction hypothesis is valid for $k = 2$.

\item We assume the induction hypothesis is valid for $k = n $ and thus
\begin{eqnarray}
\begin{array}{c}
\displaystyle

\partial_t^{(k)} \mathbf{Q} = \sum_{l=1}^{k} \mathbf{C}(k, l) \partial_x^{(l)} \mathbf{Q} + \partial_t^{(k-2)}( \mathbf{B}\partial_t  \mathbf{Q})\;,

\end{array}
\end{eqnarray}
for all $k\leq n$.

\item Let us prove this is valid for $k = n+1$. In fact
\begin{eqnarray}
\label{eq:DerivativeProof:eq-1}
\begin{array}{ll}
\displaystyle

\partial_t^{(n+1)} \mathbf{Q} & =  
\displaystyle
\sum_{l=1}^{n} 
\partial_t ( \mathbf{C}(n, l) \partial_x^{(l)} \mathbf{Q}   ) + \partial_t^{(n-1)}( \mathbf{B}\partial_t  \mathbf{Q})
\\
& = 
\displaystyle
\sum_{l=1}^{n} 
\mathbf{C}(n, l)_t \partial_x^{(l)} \mathbf{Q}   
+
\sum_{l=1}^{n} 
\mathbf{C}(n, l) \partial_x^{(l)} ( \partial_t \mathbf{Q}   ) 
+ \partial_t^{(n-1)}( \mathbf{B}\partial_t  \mathbf{Q})

\\
& = 
\displaystyle
\sum_{l=1}^{n} 
\mathbf{C}(n, l)_t \partial_x^{(l)} \mathbf{Q}   
+
\sum_{l=1}^{n} 
\mathbf{C}(n, l) 

\sum_{ m = 1}^{ l + 1} \mathbf{D}(l+1, m ) \partial_x^{(m)} \mathbf{Q}

+ \partial_t^{(n-1)}( \mathbf{B}\partial_t  \mathbf{Q})

\\
& = 
\displaystyle

\sum_{l=1}^{n} 
\mathbf{C}(n, l)_t \partial_x^{(l)} \mathbf{Q}   
+
\sum_{m=1}^{n} 

\mathbf{C}(n, m) 

\sum_{ l = 1}^{ m + 1}
 \mathbf{D}(m+1, l ) \partial_x^{(l)} \mathbf{Q}

+ \partial_t^{(n-1)}( \mathbf{B}\partial_t  \mathbf{Q})

\\
& = 
\displaystyle

\sum_{l=1}^{n} 
\biggl[
\mathbf{C}(n, l)_t    
+
\sum_{m=l-1}^{n} 

\mathbf{C}(n, m) 
\mathbf{D}(m+1, l )
\biggr]
 \partial_x^{(l)} \mathbf{Q}

\\
& 
\displaystyle
+
\mathbf{C}(n, n) 
\mathbf{D}(n+1, n+1 )
\biggr]
 \partial_x^{(n+1)} \mathbf{Q}

+ \partial_t^{(n-1)}( \mathbf{B}\partial_t  \mathbf{Q})

\;,

\end{array}
\end{eqnarray}
with  $\mathbf{C}(n, 0 ) = \mathbf{0}$. So by collecting terms and defining
\begin{eqnarray}
\begin{array}{c}

\mathbf{C}(n+1, l ) =
\left\{
\begin{array}{cc}

\mathbf{C}(n,n) \mathbf{D}(n+1, n+1) &, l = n+1 \;, \\

\mathbf{C}(n, l)_t    
+
\sum_{m=l-1}^{n} 

\mathbf{C}(n, m) 
\mathbf{D}(m+1, l ) &, l < n + 1 \;,

\end{array}

\right.
\end{array}
\end{eqnarray}
we can write (\ref{eq:DerivativeProof:eq-1}) as
\begin{eqnarray}
\begin{array}{c}

\partial_t^{(n+1)} \mathbf{Q} =  
\displaystyle
\sum_{l=1}^{n+1}

\mathbf{C}(n+1, l) \partial_x^{(l)} \mathbf{Q} 

+ \partial_t^{(n-1)}( \mathbf{B}\partial_t  \mathbf{Q})

\;, 

\end{array}
\end{eqnarray}
this proves the sought result.
\end{itemize} 

 Notice that the condition $\partial_t^{(-1)}( \mathbf{B}\partial_t  \mathbf{Q}) = \partial_t^{(-1)}(  \partial_t \mathbf{S}( \mathbf{Q}) ) = \partial_t^{(0)}\mathbf{S}(\mathbf{Q}) = \mathbf{S}(\mathbf{Q})$ is natural, which also justifies the expression (\ref{eq:hip-1}), it also corresponds to $k = 1$ in the formula (\ref{eq:ReduceCK-formula}).

\end{proof} 
Notice that the expression (\ref{eq:ReduceCK-formula}) is only possible from the simplifications proposed in this work. This is not possible, in general, for the conventional Cauchy-Kowalewskaya procedure. Notice that (\ref{eq:ReduceCK-formula}) expresses the time derivatives in terms of spatial derivatives of $\mathbf{Q}$, space and time derivatives of both $\mathbf{A}$ and $\mathbf{B}$.

\begin{corollary}\label{corollary-1}

The expression (\ref{eq:ReduceCK-formula}) can be written as
\begin{eqnarray}
\label{eq:ReduceCK-formula-useful}
\begin{array}{c}
\displaystyle

\partial_t^{(k)} \mathbf{Q} = \mathbf{M}_k +  \mathbf{B}\partial_t^{(k-1)}  \mathbf{Q}\;,

\end{array}
\end{eqnarray}
where
\begin{eqnarray}
\label{eq:2:ReduceCK-formula-useful}
\begin{array}{c}
\displaystyle

\mathbf{M}_k = 
\sum_{l=1}^{k} \mathbf{C}(k, l) \partial_x^{(l)} \mathbf{Q} 
 + 
\sum_{l=1}^{k-2} 
\left(
\begin{array}{c}
k-2\\
l-1
\end{array}
\right)
\mathbf{B}_t^{(k-1-l)} \partial_t^{(l)} \mathbf{Q}

\;.
\end{array}
\end{eqnarray}
\end{corollary}

\begin{proof}
This result follows from the manipulation of (\ref{eq:ReduceCK-formula}) in Proposition \ref{proposition:main-1}.  Particularly the term $ \partial_t^{(k-2 )} (\mathbf{B} \partial_t \mathbf{Q}) $ can be expressed, by using the result in Proposition \ref{prop:algebraic-manipulation}, as 
\begin{eqnarray}
\begin{array}{c}

\partial_t^{(k-2 )} (\mathbf{B} \partial_t \mathbf{Q} ) = \sum^{k-1}_{l=1} 
\left( 
\begin{array}{c}
k-2 \\
l-1
\end{array}
\right)
\mathbf{B}_t^{(k-1-l)} \partial_t^{(l)} \mathbf{Q}\;.
\end{array}
\end{eqnarray}
By collecting terms and isolating for $l=k-1$, we obtain

\begin{eqnarray}
\begin{array}{c}
\displaystyle

\partial_t^{(k)} \mathbf{Q} = \sum_{l=1}^{k} \mathbf{C}(k, l) \partial_x^{(l)} \mathbf{Q} 
 + 
\sum_{l=1}^{k-2} 
\left(
\begin{array}{c}
k-2\\
l-1
\end{array}
\right)
\mathbf{B}_t^{(k-1-l)} \partial_t^{(l)} \mathbf{Q}

 +  \mathbf{B}\partial_t^{(k-1)}  \mathbf{Q}\;

\end{array}
\end{eqnarray}
and thus the result holds.
\end{proof}

\begin{proposition}\label{proposition:ReduceCK-formula-general}

\begin{eqnarray}
\label{eq:ReduceCK-formula-general}
\begin{array}{c}
\displaystyle

\partial_t^{(k)} \mathbf{Q} = \sum_{r=2}^{k} \mathbf{M}_r +  \mathbf{B}^{k-1}\mathbf{S} (  \mathbf{Q} )\;,

\end{array}
\end{eqnarray}
where $\mathbf{M}_r$ are those in (\ref{eq:2:ReduceCK-formula-useful}).
\end{proposition}

\begin{proof}
This is a consequence of the corollary \ref{corollary-1}.

\end{proof}

As will be seen in next sections, the previous results provide the closed form for approximations to Cauchy Kowalewskaya functionals, $\mathbf{G}^{(k)}$, which will be important to design fixed-point iteration procedures. The appendix \ref{sec:FortranCodes} shows operational details for generating the matrix $\mathbf{C}(k,l)$ involved into the simplified Cauchy-Kowaleskaya procedure introduced in this section.

\subsection{The predictor step based on a modified implicit Taylor series expansion}

Notice that the predictor $\mathbf{Q}_i$, within $I_i^n$, is required for evaluating integrals in (\ref{eq:flux-source}). On the other hand, the evaluation of these integrals is carried out by means of quadrature rules in space and time. So, for the temporal integration we use the Gaussian rule, which involves $\tau_j$, $j=1,...,n_T$ Gaussian points. Whereas, for the spatial integration, we use the Newton-Cotes rule, which involves $\xi_m$, $m=1,...,n_S$ equidistant quadrature points. See appendix \ref{sec:set-up-nodes-approx} for further details about the set up of these quadrature points through reference elements.
This set of quadrature points allows us to build space-time nodal points $(\xi_m, \tau_j)$ within the space-time cell $I_i^n$, as illustrated in the figure \ref{fig:Node-distribution}. So, from the previous comment it is evident that for flux and source evaluations, we only need the information of $\mathbf{Q}_i$ at $(\xi_m, \tau_j)$.

To obtain approximation of the predictor at every space-time node $(\xi_m, \tau_j)$, we propose the following strategy.

\begin{enumerate}
\item Provide a starting guess for $\mathbf{Q}_i(\xi_m, \tau_j)$, $m=1,...,n_S$ and $j = 1,...,n_T$.

This is done by using the formula
\begin{eqnarray}
\begin{array}{c}
\mathbf{Q}_i(\xi_m, \tau_j) = [ \mathbf{I} - \tau_j \mathbf{B}(\mathbf{W}(\xi_m))]^{-1}
(\mathbf{W}_i(\xi_m) - \tau_j \mathbf{A}(\mathbf{W}(\xi_m) )  \partial_x \mathbf{Q}_i(\xi_m, \tau_j) )\;,
\end{array}
\end{eqnarray}
which corresponds to the second order accurate expression in \cite{Toro:2015a}. Here, $\mathbf{W}_i(\xi)$ represents the reconstruction polynomial obtained within the space-time cell $I_i^n$. Any reconstruction procedure can be implemented, however, in this work we use the Weighted Essentially Non-Oscillatory (WENO) reconstruction method described in \cite{Dumbser:2008a}.

\item Compute the approximation of high-order derivatives in time and space as well of the state and Jacobian matrices. For this purpose, we use the following approach.  

Let $\mathbf{M}$ be a function, which may represent the state function $\mathbf{Q}$ and the Jacobian matrices $\mathbf{A}$ and $\mathbf{B}$ as well. 

\begin{itemize}
\item Then, for obtaining $\mathbf{M}_x^{(l)}(\xi, \tau_j)$, we first interpolate the function $\mathbf{M}$ on the nodes $(\xi_m, \tau_j)$ with $j$ fix and varying $m=1,...,n_S$. So an interpolation function $\tilde{\mathbf{M}}(\xi, \tau_j)$ is obtained. Then, we are able to provide approximations of spatial derivatives of $\mathbf{M}(\xi,\tau_j)$ for any order $l$  by using the spatial derivatives of  $\tilde{\mathbf{M}}(\xi, \tau_j)$. 

\item Similarly, to obtain $\mathbf{M}_t^{(l)}(\xi_m, \tau)$,  we first interpolate the function $\mathbf{M}$ on the nodes $(\xi_m, \tau_j)$ with $m$ fix and varying  $j=1,...,n_T$. So an interpolation function $\tilde{\mathbf{M}}(\xi_m, \tau)$ is obtained. Then, we are able to provide approximations of temporal derivatives of $\mathbf{M}(\xi_m, \tau)$ for any order $l$  by using the temporal derivatives of  $\tilde{\mathbf{M}}(\xi_m, \tau)$. 

\end{itemize}

In the appendix \ref{sec:set-up-nodes-approx}, is shown the form of these interpolation polynomials for the orders of accuracy considered in this paper.

\item Update $\mathbf{Q}_i$ at every $(\xi_m, \tau_j)$ by using
\begin{eqnarray}
\label{eq:implicit-taylor}
\begin{array}{c}
\mathbf{Q}_i(\xi_m,\tau_j ) = \mathbf{W}_i( \xi_m) - \sum_{k=1}^{M} \frac{(-\tau_j)^k}{k!} \tilde{\mathbf{G}}^{(k)}(\xi_m,\tau_j)\;,
\end{array}
\end{eqnarray}
where $\tilde{\mathbf{G}}^{(k)} = \tilde{\mathbf{G}}^{(k)} (\mathbf{Q}_i, ...,\partial_x^{(k)}\mathbf{Q}_i, \mathbf{A}_t^{(l)},..,\mathbf{A}_x^{(l)},..,\mathbf{B}_t^{(l)},...,\mathbf{B}_x^{(l)},... ) $ is given by (\ref{eq:ReduceCK-formula}). The derivatives of the state function and matrices are evaluated at $(\xi_m, \tau_j)$ and computed in the previous step.  The equation (\ref{eq:implicit-taylor}) corresponds to the implicit Taylor series expansion in \cite{Toro:2015a} with the difference that $\tilde{\mathbf{G}}^{(k)} $ is a simplification of the conventional Cauchy-Kowalewskaya functional.

For solving (\ref{eq:implicit-taylor}), we use the following nested Picard iteration procedure.
\begin{eqnarray}
\label{eq:Nested-Picard}
\begin{array}{lcl}
\displaystyle
\mathbf{Q}_i^{s+1} & = & \mathbf{W}_i( \xi_m)  -  \sum_{k=1}^{M} \frac{(-\tau_j)^k}{k!} \sum_{l=2}^{k} \mathbf{M}_l \\
\\
\displaystyle
&-& \sum_{k=1}^{M} \frac{(-\tau_j)^k}{k!} \mathbf{B}^{k-1}(\mathbf{Q}_{i}^{s}) \mathbf{S} (\mathbf{Q}_i^{s+1} )  \;,
\end{array}
\end{eqnarray}
where $s$ is an iteration index and $ \mathbf{M}_l$ comes from the proposition \ref{proposition:ReduceCK-formula-general}. We have omitted the arguments of $\mathbf{Q}_i$.

To solve it, we build an algebraic system, which has the form
\begin{eqnarray}
\label{eq:Algebraic-system}
\begin{array}{lcl}
\displaystyle
\mathcal{H}(\mathbf{Y}) & = & \mathbf{Y} - \mathbf{W}_i( \xi_m)   +  \sum_{k=1}^{M} \frac{(-\tau_j)^k}{k!} \sum_{l=2}^{k} \mathbf{M}_l \\
\\
\displaystyle
& & +  \sum_{k=1}^{M} \frac{(-\tau_j)^k}{k!} \mathbf{B}^{k-1}(\mathbf{Q}_{i}^{s}(\xi_m, \tau_j)) \mathbf{S} (\mathbf{Y} )  \;.
\end{array}
\end{eqnarray}
So, the update of $\mathbf{Q}_i$ is carried out as $\mathbf{Q}_i^{s+1} = \mathbf{Q}_i^{s} - \delta$, where $\delta$ is the solution to $\mathcal{J}(\mathbf{Q}_i^s ) \delta = \mathcal{H}(\mathbf{Q}_i^s)$, where
\begin{eqnarray}
\label{eq:jacobian-alg-syste}
\begin{array}{c}
\displaystyle
\mathcal{J}(\mathbf{Y}) =  \mathbf{I} +  \sum_{k=1}^{M} \frac{(-\tau_j)^k}{k!} \mathbf{B}^{k-1}(\mathbf{Q}_{i}^{s}(\xi_m, \tau_j)) \mathbf{B} (\mathbf{Y} )  \;,
\end{array}
\end{eqnarray}
is the Jacobian matrix of $ \mathcal{H}(\mathbf{Y}) $ with respect to $\mathbf{Y}$.
The update is carried out $M$ times, where $M+1$ corresponds to the order of accuracy.  Notice that the same algebraic equation for $\mathbf{Y}$ has to be solved for any order of accuracy. So, the size of $\mathbf{Y}$ does not depend on the accuracy.

\item Go to step 2. Finish if the global loop has been done by $M$ times.  By virtue of the efficiency we use a limited number of iterations. From experiments, not shown here, the result does not vary in terms of accuracy if a stop criterion, based on the tolerance for the relative error between subsequent approximations, is implemented.

\end{enumerate}

Once $\mathbf{Q}_i$ is computed for each cell $I_i^n$, the numerical flux and source terms can be easily evaluated. In the appendix \ref{sec:numerical-flux-source}, is shown the form in which the integrals in (\ref{eq:flux-source}) are evaluated.

This completes the description of the proposed strategy for obtaining the predictor within the computational cell $I_i^n$ by using the implicit GRP approach.

\begin{figure}
\begin{center}
\includegraphics[scale=0.7]{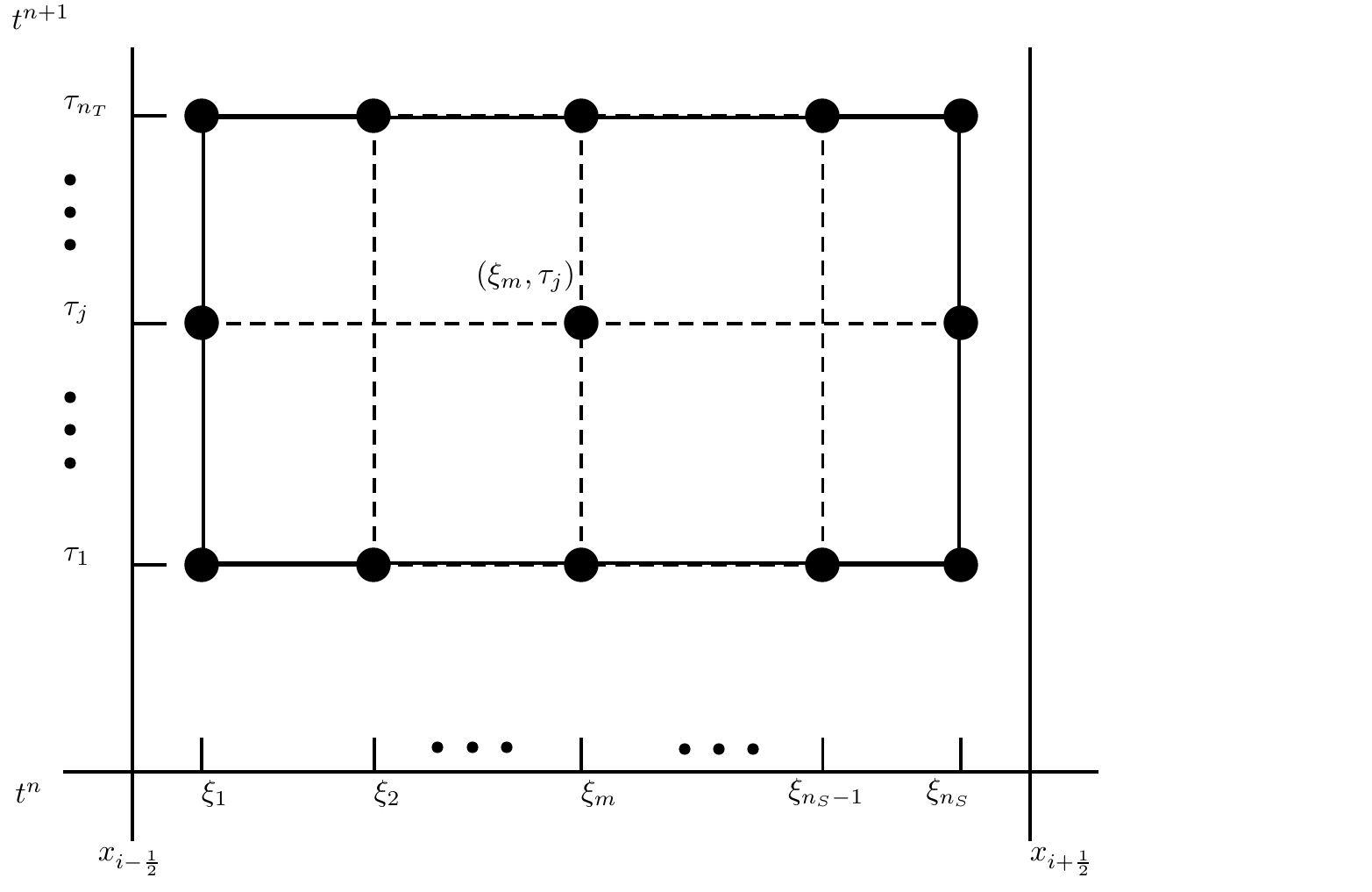}
\end{center}
\caption{Sketch of the space-time node distribution.}\label{fig:Node-distribution}
\end{figure}

\section{Numerical results}\label{section:Results}
In this section we shall consider numerical test aimed at assessing the accuracy and performance of the present scheme. The time step, $\Delta t$, will be computed by using the well-known CFL condition 
\begin{eqnarray}
\begin{array}{c}
\Delta t = C_{cfl}  \frac{ \Delta x}{ \lambda_{abs}} \;,
\end{array}
\end{eqnarray}
where $\lambda_{abs} = \max_i(  \max_j (  | \lambda_j (\mathbf{Q}_i^n) | ) ) $, here  $\lambda_j$, $j=1,...,m$ are the eigenvalues of the Jacobian matrix of $\mathbf{F}$ evaluated at $ \mathbf{Q}_i^n $, the data at each cell $I_i^n$ and the maximum is taken over all cells $I_i^n$.  On the hand, to assess the empirically the convergence rate, we are going to use the norm
\begin{eqnarray}
\label{eq:norm-p}
\begin{array}{c}

|| \mathbf{Q} - \mathbf{Q}^e||_p^p = \displaystyle \sum_{i = 1}^{N} \int_{x_{i-\frac{1}{2}}}^{x_{i+\frac{1}{2}} } |\mathbf{W}_i (x) - \mathbf{Q}^e(x, t_{End}) |^p dt \;, 

\end{array}
\end{eqnarray}
where $ \mathbf{Q}^e(x,t)$ is the exact solution, $\mathbf{W}_i(x)$ is the reconstruction polynomial within the interval $[x_{i-\frac{1}{2}}, x_{i+\frac{1}{2}} ]$ obtained at the output time of the global simulation. In this paper we are going to use (\ref{eq:norm-p}) with  $p =1$, $p=2$ and the maximum norm given by 
\begin{eqnarray}
\begin{array}{c}

|| \mathbf{Q} - \mathbf{Q}^e|| = \max_{ i} \{ \max_{ x\in [ x_{i-\frac{1}{2}},x_{i+\frac{1}{2}} ]}|\mathbf{W}_i ( x) - \mathbf{Q}^e(x, t_{End}) | \} \;.

\end{array}
\end{eqnarray}
Let us remark that, in the case of hyperbolic systems in which the solution vector $\mathbf{Q}$ contains more than one variables, we carry out the computation of errors as indicated above but for some particular component of the solution vector.
\subsection{A linear system of hyperbolic balance laws}
Here, we consider the linear system in \cite{Montecinos:2014c}, given by
\begin{eqnarray}
\begin{array}{c}
\partial_t \mathbf{Q}(x,t) + \mathbf{A} \partial_x ( \mathbf{Q} (x,t)) = \mathbf{B} \mathbf{Q} (x,t) \;, x \in [0,1]\;,\\
\mathbf{Q}(x,0) = 
\left[ 
\begin{array}{c}

\sin( 2 \pi x ) \\
\cos( 2 \pi x ) 

\end{array}
\right]
\;,
\end{array}
\end{eqnarray}
where 
\begin{eqnarray}
\begin{array}{c}

\mathbf{A} = \left[
\begin{array}{cc}
0       & \lambda \\
\lambda & 0 \\
\end{array}
\right]\;,

\mathbf{B} = \left[
\begin{array}{cc}
\beta   & 0 \\
0 & \beta \\
\end{array}
\right]
\;.
\end{array}
\end{eqnarray}

The problem is endowed with periodic boundary conditions.  
This system has the exact solution 
\begin{eqnarray}
\begin{array}{c}
\displaystyle \mathbf{Q}^{e}(x,t) = 
\frac{  e^{\beta t} }{2} 
\left[
\begin{array}{c}
\Phi(x,t ) + \Psi(x,t )  \\
\Phi(x,t ) - \Psi(x,t )  \\
    
\end{array}
\right] \;,
\end{array}
\end{eqnarray}
where
\begin{eqnarray}
\begin{array}{c}

\Phi(x,t ) = \sin( 2\pi (x - \lambda t) ) + \cos( 2\pi (x - \lambda t) )  \;, \\

\Psi(x,t) = \sin( 2\pi (x + \lambda t) ) - \cos( 2\pi (x + \lambda t) )\;.
\end{array}
\end{eqnarray}
Here we consider $\lambda =1$ and $\beta = -1$. This is a simple test aimed at evaluating the accuracy of the present scheme. As can be seen in the Table \ref{Table-LinearSystem}, the expected theoretical orders of accuracy are achieved. 
\begin{table}
\begin{center}
Theoretical order : 2 \\
\begin{tabular}{cccccccc} 
\\
\hline
\hline 
Mesh  & $L_\infty$ - err & $L_\infty$- ord  & $L_1$ - err & $L_1$ - ord & $L_2$ - err & $L_2$ - ord & CPU  \\  
\hline
     8  &  -     &$  2.45e-0 2$&  -     &$  1.33e-0 2$&  -     &$  1.74e-0 2$  &0.0064\\
    16  &  2.35  &$  4.81e-0 3$&  2.38  &$  2.55e-0 3$&  2.55  &$  2.97e-0 3$  &0.0114\\
    32  &  1.67  &$  1.52e-0 3$&  2.85  &$  3.54e-0 4$&  2.34  &$  5.84e-0 4$  &0.0287\\
    64  &  2.95  &$  1.96e-0 4$&  3.71  &$  2.71e-0 5$&  3.47  &$  5.29e-0 5$  &0.0853\\
   128  &  1.49  &$  6.96e-0 5$&  1.83  &$  7.62e-0 6$&  1.91  &$  1.41e-0 5$  &0.3395\\

 \hline
 \\
 \end{tabular}  
\\
Theoretical order : 3 \\
\begin{tabular}{cccccccc}  
\\
\hline 
Mesh  & $L_\infty$ - err & $L_\infty$- ord  & $L_1$ - err & $L_1$ - ord & $L_2$ - err & $L_2$ - ord & CPU  \\  
\hline
     8  &  -     &$  1.52e-0 2$&  -     &$  1.07e-0 2$&  -     &$  1.15e-0 2$  &0.0116\\
    16  &  2.85  &$  2.11e-0 3$&  2.97  &$  1.36e-0 3$&  2.94  &$  1.50e-0 3$  &0.0352\\
    32  &  3.01  &$  2.62e-0 4$&  3.03  &$  1.66e-0 4$&  3.02  &$  1.85e-0 4$  &0.1096\\
    64  &  3.05  &$  3.17e-0 5$&  3.04  &$  2.02e-0 5$&  3.05  &$  2.24e-0 5$  &0.4506\\
   128  &  3.01  &$  3.93e-0 6$&  3.01  &$  2.50e-0 6$&  3.01  &$  2.78e-0 6$  &2.0173\\

 \hline
 \\
 \end{tabular}  
\\
Theoretical order : 4 \\
\begin{tabular}{cccccccc}  
\\
\hline 
Mesh  & $L_\infty$ - err & $L_\infty$- ord  & $L_1$ - err & $L_1$ - ord & $L_2$ - err & $L_2$ - ord & CPU  \\  
\hline
     8  &  -     &$  7.69e-0 3$&  -     &$  4.80e-0 3$&  -     &$  5.26e-0 3$  &0.0358\\
    16  &  2.92  &$  1.02e-0 3$&  3.15  &$  5.42e-0 4$&  3.10  &$  6.14e-0 4$  &0.1280\\
    32  &  3.64  &$  8.15e-0 5$&  3.67  &$  4.26e-0 5$&  3.67  &$  4.82e-0 5$  &0.4703\\
    64  &  3.87  &$  5.56e-0 6$&  3.87  &$  2.92e-0 6$&  3.87  &$  3.30e-0 6$  &1.7537\\
   128  &  3.95  &$  3.59e-0 7$&  3.95  &$  1.89e-0 7$&  3.95  &$  2.14e-0 7$  &6.9353\\

 \hline
 \\
 \end{tabular}  
\\
Theoretical order : 5 \\
\begin{tabular}{cccccccc}  
\\
\hline 
Mesh  & $L_\infty$ - err & $L_\infty$- ord  & $L_1$ - err & $L_1$ - ord & $L_2$ - err & $L_2$ - ord & CPU  \\  
\hline

     8  &  -     &$  1.49e-0 3$&  -     &$  6.42e-0 4$&  -     &$  7.65e-0 4$  &0.1270\\
    16  &  4.85  &$  5.16e-0 5$&  4.88  &$  2.18e-0 5$&  4.89  &$  2.59e-0 5$  &0.3801 \\
    32  &  4.96  &$  1.66e-0 6$&  4.96  &$  6.99e-0 7$&  4.96  &$  8.29e-0 7$  &1.4746 \\
    64  &  4.99  &$  5.23e-0 8$&  4.99  &$  2.19e-0 8$&  4.99  &$  2.60e-0 8$  &5.7874 \\
   128  &  5.00  &$  1.64e-0 9$&  4.99  &$  6.88e-010$&  5.00  &$  8.16e-010$  &22.4536 \\

 \hline
\end{tabular} 
\end{center}
\caption{Linear system. Output time
  $t_{out} = 1$ with  $C_{cfl}= 0.9,$ $\beta = -1$, $\lambda =  1 $.}\label{Table-LinearSystem}
\end{table}

 \subsection{A system of non-linear hyperbolic balance laws}
 
Here we assess the present methods, applied to the non-linear system 
\begin{eqnarray}
  \label{eq:1:test2}
  \begin{array}{c}
  \partial_t \mathbf{Q} + \partial_x \mathbf{F}(\mathbf{Q}) = \mathbf{S}(\mathbf{Q})\;,  \\  
    \mathbf{Q} = 
    \left[ \sin(2 \pi x), \cos ( 2 \pi x) \right]^T\;,
  \end{array}
\end{eqnarray}
where $\mathbf{F}(\mathbf{Q})$ and $\mathbf{S}(\mathbf{Q})$ are given by 
\begin{eqnarray}
\label{eq:nonlinear:CIC-2}
\begin{array}{cc}
\mathbf{F}(\mathbf{Q}) = 
\left[
\begin{array}{c}
\frac{1}{9}\biggl( \frac{5}{2}u^2+v^2-uv \biggr)\\
\frac{1}{9}\biggl(  4uv-u^2+\frac{1}{2}v^2 \biggr)
\end{array}
\right]\;,
&
\mathbf{S}(\mathbf{Q}) = 
\left[
\begin{array}{c}
\beta \biggl( \frac{2u-v}{3} \biggr)^2\\
-\beta \biggl( \frac{2u-v}{3} \biggr)^2
\end{array}
\right]\;,
\end{array}
\end{eqnarray}
where $\beta \leq 0$ is a constant value, see \cite{Toro:2015a}. The exact solution is given by
\begin{eqnarray}
  \label{eq:7:test2}
  \begin{array}{ccc}
  u(x,t) &=& w_1(x,t)+w_2(x,t)\;,\\
  v(x,t) &=& 2w_1(x,t)-w_2(x,t)\;,\\  
  \end{array}
\end{eqnarray}
where $w_1$ and $w_2$ are the solutions  to 
\begin{eqnarray}
  \label{eq:6:test6}
  \begin{array}{ccc}
    \partial_t w_1 +w_1\partial_x (w_1) &=& 0\;,  \\
    \partial_t w_2 +w_2\partial_x (w_2) &=& \beta w_2^2\;,  \\
  \end{array}
\end{eqnarray}
where the initial condition for each equation is  
$$ w_1(x,0) = \frac{\sin(2\pi x)+cos(2\pi x)}{3} \;, w_2(x,0) = \frac{2sin(2 \pi x)-cos(2\pi x)}{3}\;.$$ 

Notice that system (\ref{eq:6:test6}) requires the solution of the Burgers equation with a non linear source term, in \cite{Toro:2015a} this solution is reported. Table \ref{Table-NonLinearSystem}, shows the empirical orders of accuracy and the CPU times. Comparing with CPU times of the implicit Taylor series expansion and conventional Cauchy-Koealewskaya procedure in \cite{Toro:2015a}, we observe that the present scheme depicts important improvements in the performance. An improvement of one order of magnitude compared with the strategies in \cite{Toro:2015a}, is obtained. 																													%
\begin{table}
\begin{center}
Theoretical order : 2 \\
\begin{tabular}{cccccccc} 
\\
\hline
\hline 
Mesh  & $L_\infty$ - err & $L_\infty$- ord  & $L_1$ - err & $L_1$ - ord & $L_2$ - err & $L_2$ - ord & CPU  \\  
\hline

    32  &  -     &$  1.95e-0 2$&  -     &$  5.45e-0 3$&  -     &$  7.81e-0 3$  &0.0061\\
    64  &  1.53  &$  6.75e-0 3$&  2.12  &$  1.25e-0 3$&  1.94  &$  2.04e-0 3$  &0.0201\\
   128  &  1.75  &$  2.01e-0 3$&  2.30  &$  2.54e-0 4$&  2.18  &$  4.49e-0 4$  &0.0490\\
   256  &  1.80  &$  5.79e-0 4$&  2.18  &$  5.62e-0 5$&  2.15  &$  1.01e-0 4$  &0.1703\\
   512  &  1.38  &$  2.23e-0 4$&  2.13  &$  1.29e-0 5$&  2.06  &$  2.44e-0 5$  &0.7851\\
																																												
\hline
\\
\end{tabular}  
\\
Theoretical order : 3 \\
\begin{tabular}{cccccccc}  
\\
\hline 
Mesh  & $L_\infty$ - err & $L_\infty$- ord  & $L_1$ - err & $L_1$ - ord & $L_2$ - err & $L_2$ - ord & CPU  \\  
\hline

    32  &  -     &$  2.00e-0 2$&  -     &$  2.88e-0 3$&  -     &$  5.17e-0 3$  &0.0139\\
    64  &  2.47  &$  3.63e-0 3$&  2.69  &$  4.45e-0 4$&  2.57  &$  8.67e-0 4$  &0.0668\\
   128  &  2.74  &$  5.42e-0 4$&  2.92  &$  5.90e-0 5$&  2.85  &$  1.20e-0 4$  &0.1935\\
   256  &  2.84  &$  7.59e-0 5$&  2.95  &$  7.62e-0 6$&  2.93  &$  1.57e-0 5$  &0.7559\\
   512  &  2.95  &$  9.83e-0 6$&  2.99  &$  9.62e-0 7$&  2.98  &$  2.00e-0 6$  &2.9953\\

 \hline
\\
\end{tabular}  
\\
Theoretical order : 4 \\
\begin{tabular}{cccccccc}  
\\
\hline 
Mesh  & $L_\infty$ - err & $L_\infty$- ord  & $L_1$ - err & $L_1$ - ord & $L_2$ - err & $L_2$ - ord & CPU  \\  
\hline

    32  &  -     &$  2.44e-0 2$&  -     &$  3.33e-0 3$&  -     &$  6.09e-0 3$  &0.0609\\
    64  &  3.07  &$  2.90e-0 3$&  3.53  &$  2.89e-0 4$&  3.24  &$  6.46e-0 4$  &0.2186\\
   128  &  3.88  &$  1.97e-0 4$&  4.14  &$  1.64e-0 5$&  4.00  &$  4.03e-0 5$  &0.7516\\
   256  &  4.28  &$  1.01e-0 5$&  4.40  &$  7.80e-0 7$&  4.34  &$  1.99e-0 6$  &3.0089\\
   512  &  4.18  &$  5.58e-0 7$&  4.33  &$  3.87e-0 8$&  4.33  &$  9.89e-0 8$  &11.5548\\

 \hline
\\
\end{tabular}  
\\
Theoretical order : 5 \\
\begin{tabular}{cccccccc}  
\\
\hline 
Mesh  & $L_\infty$ - err & $L_\infty$- ord  & $L_1$ - err & $L_1$ - ord & $L_2$ - err & $L_2$ - ord & CPU  \\  
\hline

    32  &  -     &$  8.80e-0 3$&  -     &$  8.56e-0 4$&  -     &$  1.82e-0 3$  &0.1879\\
    64  &  3.48  &$  7.91e-0 4$&  4.07  &$  5.09e-0 5$&  3.82  &$  1.29e-0 4$  &0.6927\\
   128  &  4.46  &$  3.60e-0 5$&  4.66  &$  2.02e-0 6$&  4.58  &$  5.37e-0 6$  &2.4518\\
   256  &  4.81  &$  1.29e-0 6$&  4.72  &$  7.65e-0 8$&  4.82  &$  1.90e-0 7$  &9.6558\\
   512  &  4.46  &$  5.87e-0 8$&  3.97  &$  4.86e-0 9$&  4.41  &$  8.90e-0 9$  &38.51853\\

 \hline
\end{tabular} 
\end{center}
\caption{Non-linear system. Output time
  $t_{out} = 0.1$ with  $C_{cfl}= 0.9,$ $\beta = -1$.}\label{Table-NonLinearSystem}
\end{table}

\subsection{The LeVeque and Yee test}

Here, we apply our schemes to the well-known and challenging scalar test
problem proposed by LeVeque and Yee \cite{LeVeque:1990a}, given by
\begin{eqnarray}
\begin{array}{c}
\partial_t q(x,t) + \partial_x q(x,t) = \beta q(x,t)(q(x,t) - 1)(q(x,t) - \frac{1}{2} ) \;.
\end{array}
\end{eqnarray}
We solve this PDE on the computational domain $[0,1]$ with transmissive boundary conditions and the initial condition given by
\begin{eqnarray}
\begin{array}{c}
q(x,0) = \left\{

\begin{array}{cc}
1 \;, x < 0.3 \;, \\
0 \;, x > 0.3 \;. \\
\end{array}
\right.
\end{array}
\end{eqnarray}
The solution on the characteristic curves satisfies de ordinary differential equation  $ \frac{d(x(t), t)}{dt} = \beta q(x(t),t)(q(x(t),t) - 1)(q(x(t),t) - \frac{1}{2} )  $, which has two stable solutions $q \equiv 0$ and $q \equiv 1$ and one unstable solution in $q \equiv \frac{1}{2}$ where any solution trays to away from this. Similarly, any solution associated to characteristic curves necessarily must converge to one of the two stable solutions. On the other hand, a numerical scheme which is not able to solve stiff source terms, may introduce an excessive  numerical diffusion and so the numerical solution, following characteristic curves, converges to the wrong stable solution. This penalizes the right propagation. Figure \ref{fig:LevequeAndYee} shows the comparison between the exact solution and the numerical approximations provided by the present scheme of second, third, fourth and fifth orders of accuracy. The figure shows a good agreement for $\beta = - 10000$ at $t_{out} = 0.3$, which correspond to the stiff regime. We have used $ C_{cfl}= 0.2$ and $300$ cells. This test illustrates the ability of the present scheme for solving hyperbolic balance laws with stiff source terms.
\begin{figure}
\begin{center}
\includegraphics[scale=0.5]{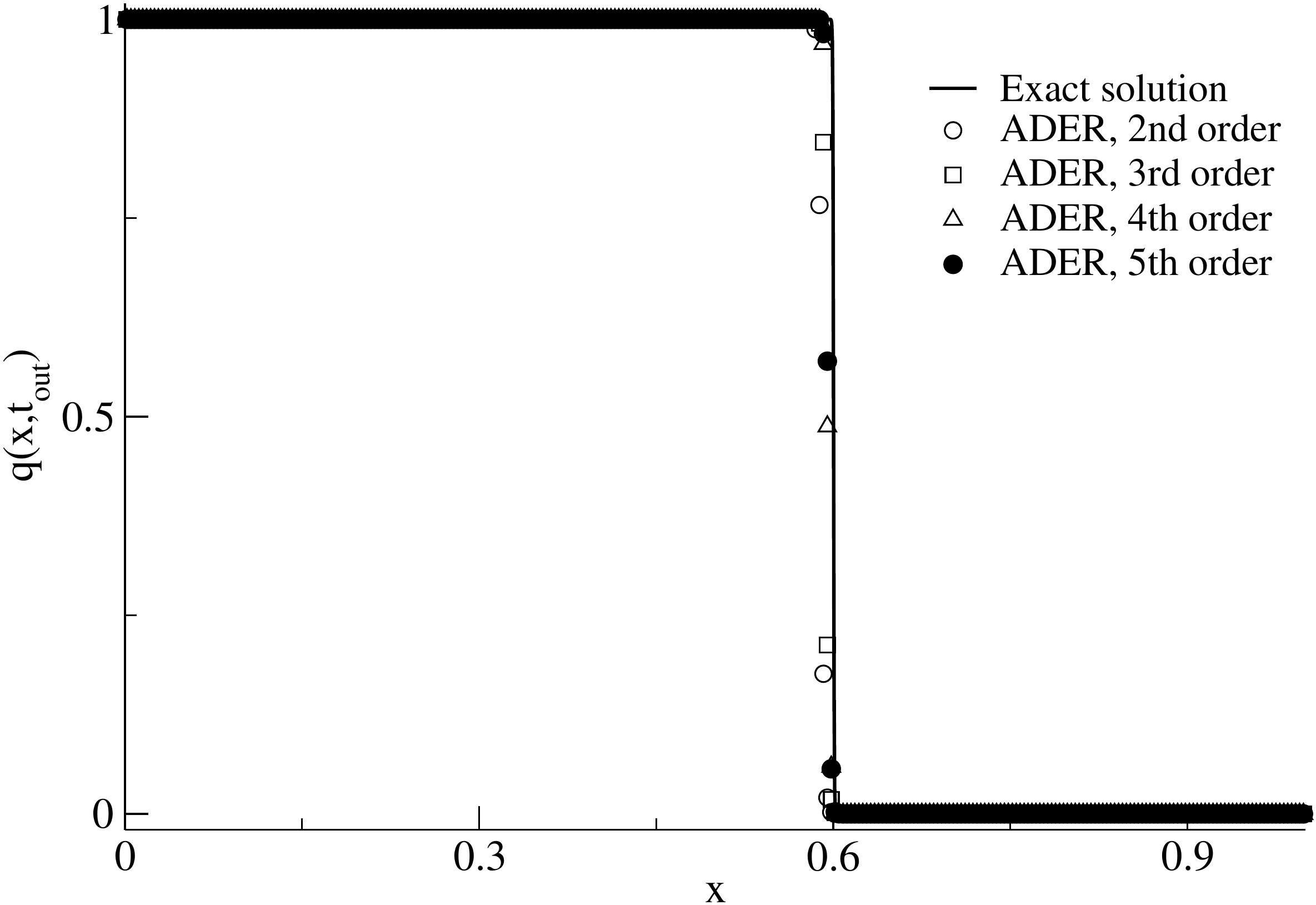}
\end{center}
\caption{Leveque and Yee test. We have used $300$ cells, $C_{cfl}=0.2$, $t_{out} = 0.3$, $\beta  = -1000$.}\label{fig:LevequeAndYee}
\end{figure}
%

\subsection{The Euler equations}
Now let us consider the Euler equations, given by
\begin{eqnarray}
\label{eq:1-euler}
\begin{array}{ccc}
\mathbf{Q} =
\left[
\begin{array}{c}
\rho \\
\rho u \\
E
\end{array}
\right]\;,
&
\mathbf{F}(\mathbf{Q}) =
\left[
\begin{array}{c}
\rho u \\
\rho u^2 + p \\
u(E+p)
\end{array}
\right]\;,
\end{array}
\end{eqnarray}
where the pressure $p$ is related with the conserved variables through the equation 
\begin{eqnarray}
\label{eq:2-euler}
p = (\gamma -1) (E -\frac{\rho u^2}{2}) \;,
\end{eqnarray}
for an ideal gas $\gamma = 1.4$.  Notice that, the choice of the initial condition given by the functions 
\begin{eqnarray}
\begin{array}{lcl}
\rho(x,0) &=& 1+0.2\sin(2\pi x)\;,\\
   u(x,0) &=& 1,\\
   p(x,0) &=& 2,
\end{array}
\end{eqnarray}
provides the exact solution for the system (\ref{eq:1-euler}), which corresponds to the set of functions 
\begin{eqnarray}
\begin{array}{ccl}
\rho(x,t) &=& 1+0.2\sin(2\pi (x-t))\;,\\
u(x,t)    &=& 1\;,\\
p(x,t)    &=& 2\;.
\end{array}
\end{eqnarray}
Notice that, the variables $\rho, u ,p$ correspond to the non-conservative variables, the corresponding translation to conserved variables needs to be done.  This test has a complex eigenstructure, which is a challenge for numerical methods. 
Table \ref{Table-Euler}, shows the results of the empirical convergence rate assessment for the density variable $\rho$, at $t_{out} = 1$ and $C_{cfl}= 0.9$, we observe that the scheme achieves the expected theoretical orders of accuracy.
\begin{table}
\begin{center}
Theoretical order : 2 \\
\begin{tabular}{cccccccc} 
\\
\hline
\hline 
Mesh  & $L_\infty$ - err & $L_\infty$- ord  & $L_1$ - err & $L_1$ - ord & $L_2$ - err & $L_2$ - ord & CPU  \\  
\hline

     8  &  -     &$  1.51e-0 1$&  -     &$  1.05e-0 1$&  -     &$  1.14e-0 1$  &0.0119\\
    16  &  1.22  &$  6.45e-0 2$&  1.76  &$  3.09e-0 2$&  1.59  &$  3.79e-0 2$  &0.0401\\
    32  &  1.43  &$  2.40e-0 2$&  1.62  &$  1.01e-0 2$&  1.65  &$  1.21e-0 2$  &0.1569\\
    64  &  1.49  &$  8.50e-0 3$&  1.90  &$  2.71e-0 3$&  1.79  &$  3.49e-0 3$  &0.4242\\
   128  &  1.53  &$  2.95e-0 3$&  2.07  &$  6.46e-0 4$&  1.85  &$  9.69e-0 4$  &1.6852\\

 \hline
 \\
 \end{tabular}  
\\
Theoretical order : 3 \\
\begin{tabular}{cccccccc}  
\\
\hline 
Mesh  & $L_\infty$ - err & $L_\infty$- ord  & $L_1$ - err & $L_1$ - ord & $L_2$ - err & $L_2$ - ord & CPU  \\  
\hline

     8  &  -     &$  8.33e-0 2$&  -     &$  5.18e-0 2$&  -     &$  5.83e-0 2$  &0.0356\\
    16  &  2.55  &$  1.43e-0 2$&  2.58  &$  8.68e-0 3$&  2.59  &$  9.68e-0 3$  &0.1292\\
    32  &  2.90  &$  1.91e-0 3$&  2.93  &$  1.14e-0 3$&  2.93  &$  1.27e-0 3$  &0.5785\\
    64  &  2.98  &$  2.42e-0 4$&  2.99  &$  1.44e-0 4$&  2.99  &$  1.60e-0 4$  &2.0123\\
   128  &  3.00  &$  3.03e-0 5$&  3.00  &$  1.80e-0 5$&  3.00  &$  2.00e-0 5$  &7.3028\\

 \hline
 \\
 \end{tabular}  
\\
Theoretical order : 4 \\
\begin{tabular}{cccccccc}  
\\
\hline 
Mesh  & $L_\infty$ - err & $L_\infty$- ord  & $L_1$ - err & $L_1$ - ord & $L_2$ - err & $L_2$ - ord & CPU  \\  
\hline

     8  &  -     &$  7.49e-0 2$&  -     &$  4.75e-0 2$&  -     &$  5.30e-0 2$  &0.1776\\
    16  &  4.17  &$  4.17e-0 3$&  4.24  &$  2.51e-0 3$&  4.25  &$  2.79e-0 3$  &0.5581\\
    32  &  4.40  &$  1.97e-0 4$&  4.49  &$  1.12e-0 4$&  4.49  &$  1.25e-0 4$  &2.0833\\
    64  &  4.27  &$  1.02e-0 5$&  4.24  &$  5.91e-0 6$&  4.24  &$  6.59e-0 6$  &7.4370\\
   128  &  4.12  &$  5.88e-0 7$&  4.08  &$  3.50e-0 7$&  4.08  &$  3.90e-0 7$  &30.9521\\

\hline
 \\
 \end{tabular}  
\\
Theoretical order : 5 \\
\begin{tabular}{cccccccc}  
\\
\hline 
Mesh  & $L_\infty$ - err & $L_\infty$- ord  & $L_1$ - err & $L_1$ - ord & $L_2$ - err & $L_2$ - ord & CPU  \\  
\hline

     8  &  -     &$  1.24e-0 2$&  -     &$  7.62e-0 3$&  -     &$  8.47e-0 3$  &0.3804\\
    16  &  4.75  &$  4.59e-0 4$&  4.80  &$  2.73e-0 4$&  4.80  &$  3.04e-0 4$  &1.6927\\
    32  &  4.95  &$  1.48e-0 5$&  4.96  &$  8.76e-0 6$&  4.96  &$  9.76e-0 6$  &6.4005\\
    64  &  4.99  &$  4.66e-0 7$&  4.99  &$  2.75e-0 7$&  4.99  &$  3.07e-0 7$  &23.9378\\
   128  &  5.00  &$  1.46e-0 8$&  5.00  &$  8.62e-0 9$&  5.00  &$  9.59e-0 9$  &88.51398\\

\hline

\end{tabular} 
\end{center}
\caption{Euler equations.Output time
  $t_{out} = 1$ with  $C_{cfl}= 0.9.$}\label{Table-Euler}
\end{table}

\subsection{  The Shu and Osher test}

Here, we consider the test problem, first time proposed by Shu and Osher in \cite{Shu:1988a}, which is given by (\ref{eq:1-euler}) and the initial condition given, in terms of non-conserved variables $\mathbf{W} = [\rho, u, p]^T$, as
\begin{eqnarray}
\begin{array}{c}
\mathbf{W}(x,0) =
\left\{  
\begin{array}{cc}
(3.8571, 2.6294, 10.333)     &, x < -0.8 \;,\\
(1 + \sin(5 \pi x), 0 , 1)   &, x \geq  -0.8 \;.
\end{array}
\right.
\end{array}
\end{eqnarray}
The problem is solved on $[-1,1]$ up to the output time $t_{out}$,  see \cite{Goetz:2016a} for further details.  Figure \ref{fig:Shu-osher} shows a comparison between a reference solution and numerical approximations. The reference solution has been obtained with the scheme of third order using $2000$ cells. The numerical results correspond to second and third orders of accuracy for which we have used $300$ cells, $C_{cfl}= 0.5$ and $t_{out}= 0.47$. This test illustrates the ability of the present scheme for solving complex fluids, a good agreement is observed for the scheme of second and third order of accuracy on the smooth region, whereas, the third order scheme present a better performance in both the smooth region and the high frequency region as well.

\begin{figure}
\begin{center}
\includegraphics[scale=0.5]{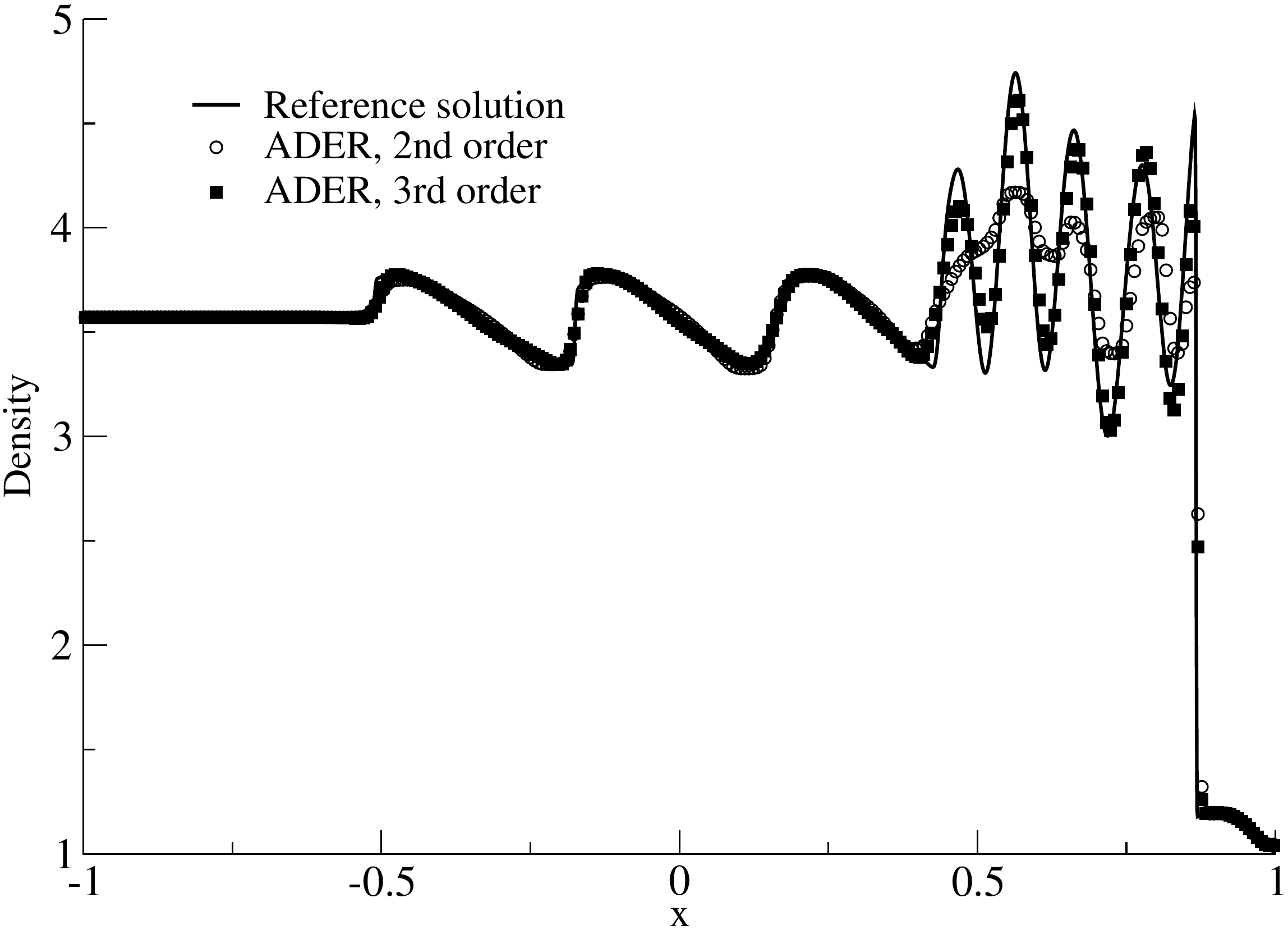}
\end{center}
\caption{The Shu-Osher test. Otput time $t_{out} = 0.47$, $300$ cells and $CFL = 0.5$.}\label{fig:Shu-osher}
\end{figure}

\section{Conclusions}\label{section:Conclusions}
In this paper, a simplified Cauchy-Kowalewkaya procedure has been proposed. The strategy uses not only the spatial derivatives of the data but also de derivatives of the Jacobian matrices in both space and time. The simplification allows us to propose a recursive formula which requires the ability of obtaining time and space derivatives of the data as well as matrices. This is achieved by using interpolations within a suitable arrangement of nodal points, which allows us to extract the information for flux and source term evaluations in a straightforward manner. This method is implemented in the context of GRP's solvers based on implicit Taylor series expansions. The solver in \cite{Toro:2015a} uses the same elements, that is, implicit Taylor series expansions and Cauchy-Kowaleskaya procedure. The Taylor expansion is used to obtain the data and the evolution of spatial derivatives required for the scheme. However, in the present approach the Taylor series expansion is used only once and the evolution of the space derivatives is not required. Despite, in both approaches, that is in \cite{Toro:2015a} and the present one, the solution of an algebraic equation is required, that in \cite{Toro:2015a} increases the number of unknowns as the accuracy increases, whereas, the number of unknowns in the present approach remains constant when the accuracy increases. We have implemented the GRP solver in the context of ADER methods and several tests reported in the literature have been solved. An empirical convergence rate assessment has been done for some of them. We have observed that the performance of the present scheme is at least one order of magnitude cheaper than schemes in \cite{Toro:2015a}. Furthermore, the expected theoretical orders of accuracy have been achieved up to fifth order of accuracy. The extension to hyperbolic systems in 2D and 3D is the subject of ongoing research.

\section*{Acknowledgements}

G.M thanks to the {\it National chilean Fund for Scientific and Technological Development}, FONDECYT, in the frame of the research project
 for Initiation in Research, number 11180926.

\bibliographystyle{plain}
\bibliography{ref} 

\appendix 

\section{Matrix-vector multiplication}\label{sec:matrix-manipulation}

In this appendix, we provide the algebraic details required for the main results in section \ref{sec:simplified-CK}.

\begin{proposition}\label{prop:algebraic-manipulation}
\begin{eqnarray}
\begin{array}{c}
\partial_t^{(l)} ( \mathbf{A} \cdot \mathbf{B}) = \sum_{k = 0}^{l}
\left(
\begin{array}{c}
l \\ 
k 
\end{array}
\right)
\mathbf{A}_t^{(l-k)} \cdot \mathbf{B}_t^{(k)}\;,
\end{array}
\end{eqnarray}
with the convention $\mathbf{A}_t^{(0)} = \mathbf{A}$,  here $ \mathbf{A}\cdot \mathbf{B}$ can be any matrix-matrix or matrix-vector multiplication. 
\end{proposition}
 
\begin{proof}
 
The proof will be carried out by mathematical induction.  Let us simplify the notation by considering matrices and vectors as elements which only depend on the variable $t$. Furthermore, we are going to assume that both $\mathbf{A}$ and $\mathbf{B}$ are regular enough in the variable $t$. 

\begin{itemize}
\item Let us verify for $l = 1$. In fact, it is asy to verify that
\begin{eqnarray}
\begin{array}{lcl}

\partial_t( \mathbf{A} \cdot \mathbf{B}) 
&=& 
\lim_{h\rightarrow 0} \frac{ \mathbf{A}(t+h) \cdot \mathbf{B}(t+h) - \mathbf{A}(t) \cdot \mathbf{B}(t)}{h}  \\
& = &  
\lim_{h\rightarrow 0} \frac{ \mathbf{A}(t+h) \cdot \mathbf{B}(t+h) - \mathbf{A}(t) \cdot \mathbf{B}(t+h) + \mathbf{A}(t) \cdot \mathbf{B}(t+h) - \mathbf{A}(t) \cdot \mathbf{B}(t)}{h}
 \\
& = &  \lim_{h\rightarrow 0} \frac{ ( \mathbf{A}(t+h)  - \mathbf{A}(t) ) \cdot \mathbf{B}(t+h)  }{h} 
     + \lim_{h\rightarrow 0} \frac{ \mathbf{A}(t) \cdot ( \mathbf{B}(t+h) - \mathbf{B}(t) ) }{h}
     
\\
& = &  \partial_t \mathbf{A} \cdot \mathbf{B} + \mathbf{A} \cdot \partial_t \mathbf{B} \;.  
     
\end{array}
\end{eqnarray}

\item Let us assume the formula is valid up to $l=n$, that is
\begin{eqnarray}
\begin{array}{c}
\partial_t^{(n)} ( \mathbf{A} \cdot \mathbf{B}) = \sum_{k = 0}^{n}
\left(
\begin{array}{c}
n \\ 
k 
\end{array}
\right)
\mathbf{A}_t^{(n-k)} \cdot \mathbf{B}_t^{(k)}
\end{array}
\end{eqnarray}
and we are going to prove, in the next step, that this is valid also for $l = n+1$.

\item Since $\mathbf{A}$ and $\mathbf{B}$ are regular enough, we have

\begin{eqnarray}
\begin{array}{c}
\partial_t^{(n+1)} ( \mathbf{A} \cdot \mathbf{B}) = \sum_{k = 0}^{n}
\left(
\begin{array}{c}
n \\ 
k 
\end{array}
\right)
\partial_t ( \mathbf{A}_t^{(n-k)} \cdot \mathbf{B}_t^{(k)} ) \\
= 
\sum_{k = 0}^{n}
\left(
\begin{array}{c}
n \\ 
k 
\end{array}
\right)
 ( \mathbf{A}_t^{(n+1-k)} \cdot \mathbf{B}_t^{(k)} + \mathbf{A}_t^{(n-k)} \cdot \mathbf{B}_t^{(k+1)} ) \;,

\end{array}
\end{eqnarray}
by expanding this expression we have
\begin{eqnarray}
\begin{array}{c}
\partial_t^{(n+1)} ( \mathbf{A} \cdot \mathbf{B}) = 

\left(
\begin{array}{c}
n \\ 
0 
\end{array}
\right) \mathbf{A}_t^{(n+1)} \cdot \mathbf{B}_t^{(0)} 

+ \left(
\left(
\begin{array}{c}
n \\ 
1 
\end{array}
\right)
+
\left(
\begin{array}{c}
n \\ 
0 
\end{array}
\right)
\right)
\mathbf{A}_t^{(n)} \cdot \mathbf{B}_t^{(1)} + \hdots
\\

+ \left(
\left(
\begin{array}{c}
n \\ 
k 
\end{array}
\right)
+
\left(
\begin{array}{c}
n \\ 
k-1 
\end{array}
\right)
\right)
\mathbf{A}_t^{(n+1-k)} \cdot \mathbf{B}_t^{(k)} + \hdots 

+
\left(
\begin{array}{c}
n \\ 
n 
\end{array}
\right)

\mathbf{A}_t^{(0)} \cdot \mathbf{B}_t^{(n+1)} \;.

\end{array}
\end{eqnarray}
By collecting terms, and using elemental combinatorial algebra we have 

\begin{eqnarray}
\begin{array}{c}
\partial_t^{(n+1)} ( \mathbf{A} \cdot \mathbf{B}) = \sum_{k = 0}^{n+1}
\left(
\begin{array}{c}
n+1 \\ 
k 
\end{array}
\right)
\mathbf{A}_t^{(n+1-k)} \cdot \mathbf{B}_t^{(k)} \;.
\end{array}
\end{eqnarray}
Thus the proof is completed.
\end{itemize}

\end{proof}

Notice that the proof is carried out for the derivative with respect to the variable $t$, however, this result applies to derivatives with respect to any variable.

\section{Operational details}\label{sec:appendix-details}
In this appendix, Fortran 90 codes of the recursive formula for obtaining time derivatives are reported. Furthermore, the set up of quadrature points through reference elements and how these are used to obtain polynomials for approximating the spatial and temporal derivatives of high-order, are provided. The numerical flux and the source term evaluation strategy is also reported. 

\subsection{Fortran codes for computing the matrix coefficients $\mathbf{D}$ and $\mathbf{C}$}\label{sec:FortranCodes}
The codes in Fortran 90 for the main expressions obtained in the section \ref{sec:simplified-CK}, are reported.

The subroutines $MATRIX\_D$ and $MATRIX\_C$, provide the matrix $D$ and $C$ corresponding to the formulas (\ref{eq:Matrix-D}) and (\ref{eq:Matrix-C}), respectively.

\begin{verbatim}
  ! ------------------------------------------------------------------------
  SUBROUTINE MATRIX_D ( l, k, i, j, A, B, DxA, DxB, D)
    ! ----------------------------------------------------------------------
    ! This subroutine computes the matrix D for the recursive simplified
    ! Cauchy-Kowalewskaya procedure in  Proposition 3.2 and Lemma 3.1. 
    ! The inputs of this subroutine are:
    !
    ! l and k which are the indices in the equation (21).
    ! i, j are the index for the space and time nodal location.
    ! DxA, DxB, the spatial derivatives of matrix A and B. 
    !
    ! This subroutine has been implemented with the following global variables: 
    ! "Accuracy" is the accuracy of the scheme, which coincides with the
    ! number of quadrature points in space.  
    ! "nGP" the number of quadrature points for time, here nGP = Accuracy-1.
    ! "NVAR" number of unknowns. 
    ! ----------------------------------------------------------------------
    DOUBLE PRECISION DxA ( NVAR, NVAR, Accuracy, nGP, Accuracy-1),       &
         & DxB ( NVAR, NVAR, Accuracy, nGP, Accuracy-1),                 &
         & A ( NVAR, NVAR, Accuracy, nGP),                               &
         & B ( NVAR, NVAR, Accuracy, nGP),                               &
         & D ( NVAR, NVAR)

    INTEGER i, j, k, l
    ! ----------------------------------------------------------------------
    ! Variables passing this point are local.
    ! ----------------------------------------------------------------------
    DOUBLE PRECISION comB, comA
    DOUBLE PRECISION DA ( NVAR, NVAR), DB ( NVAR, NVAR)
    ! ----------------------------------------------------------------------

    if (l - k - 1 < 0) then
       DB = 0.0
    elseif (l - k - 1 == 0) then
       DB = B ( :, :, i, j)
    else
       DB = DxB ( :, :, i, j , l - k - 1)                        
    end if

    !--------------------------------------------------------------------
    ! Notice that the zero-derivative of a matrix is the same matrix.
    ! Due to the form in which matrices DxA and DxB where constructed, 
    ! the last entry represents the order of the space derivative.
    !--------------------------------------------------------------------
    if (l - k  == 0) then
       DA = A ( :, :, i, j)
    else
       DA = DxA ( :, :, i, j , l - k) 
    end if
    !--------------------------------------------------------------------
    ! "FUN_COMBINATORY(n,k)" is a function which returns the combinatorial 
    ! function of n over k.
    !--------------------------------------------------------------------
    comB = FUN_COMBINATORY ( l - 2, l - k - 1)
    comA = FUN_COMBINATORY ( l - 1, l - k)
    !
    D = comB * DB - comA * DA

    ! ----------------------------------------------------------------------
  END SUBROUTINE MATRIX_D
  ! ------------------------------------------------------------------------

\end{verbatim}

\begin{verbatim}

  ! ------------------------------------------------------------------------
  SUBROUTINE MATRIX_C ( A, B, DxA, DxB, CM)
    ! ----------------------------------------------------------------------
    ! This subroutine computes the matrix "C" for the recursive simplified
    ! Cauchy-Kowalewskaya procedure in Proposition 3.2. The inputs of this 
    ! matrix are:
    ! "DxA", "DxB", the spatial derivatives of matrix A and B. These are
    ! derivatives from first order to Accuracy-th order.  
    !
    ! This subroutine has been implemented with the following global variables: 
    ! "Accuracy" is the accuracy of the scheme, which coincides with the
    ! number of quadrature points in space.  
    ! "nGP" the number of quadrature points for time, here nGP = Accuracy-1.
    ! "NVAR" number of unknowns. 
    ! ----------------------------------------------------------------------
    DOUBLE PRECISION DxA ( NVAR, NVAR, Accuracy, nGP, Accuracy-1),       &
         & DxB ( NVAR, NVAR, Accuracy, nGP, Accuracy-1),                 &
         & A ( NVAR, NVAR, Accuracy, nGP),                               &
         & B ( NVAR, NVAR, Accuracy, nGP),                               &     
         & CM ( NVAR, NVAR, Accuracy, nGP, Accuracy, Accuracy)          
    ! ----------------------------------------------------------------------
    ! Variables passing this point are local.
    ! ----------------------------------------------------------------------
    INTEGER i, j, k1, l1, m, NTIME, NSPACE

    DOUBLE PRECISION DtCM ( NVAR, NVAR, Accuracy, nGP, nGP-1, Accuracy-1, &
         & Accuracy), Dmat ( NVAR, NVAR)
    ! ----------------------------------------------------------------------

    ! ----------------------------------------------------------------------
    NTime  = nGP         ! Rename the number of quadrature points in time.
    NSpace = Accuracy    ! Rename the number of quadrature points in space.

    ! ----------------------------------------------------------------------
    ! Initialize the coefficient-matrix at zero.
    ! ----------------------------------------------------------------------
    CM = 0.0
    

    !--------------------------------------------------------------------
    ! Compute the matrices contributing to the Cauchy-Kowalewskaya functional
    ! of the simplified approach. These are the matrices which are the 
    ! matrix coefficients of spatial derivatives. See formula (29).
    !--------------------------------------------------------------------
    ! Set "C(1,1) = -A". See Proposition 3.2.
    !--------------------------------------------------------------------
    do j = 1, NTime
       do i = 1, NSpace

          CM ( :, :, i, j, 1, 1) = - A ( :, :, i, j)

       end do
    end do
    !----------------------------------------------------------------
    ! Initialize the recursive step to generate matrix coefficient C.
    !----------------------------------------------------------------     
    do k1 = 2, Accuracy - 1
       !
       do j = 1, NTime
          do i = 1, NSpace              
             !-------------------------------------------------------------
             ! The "MATRIX_D" implements the formula (21), providing "Dmat".
             !-------------------------------------------------------------             
             call MATRIX_D ( k1, k1, i, j, A, B, DxA, DxB, Dmat)
             !
             CM (  :, :, i, j, k1, k1) =  matmul (                   &
                  & CM ( :, :, i, j, k1-1, k1-1), Dmat)
          end do
       end do
       !-------------------------------------------------------------
       do l1 = 1,  k1 - 1
          !
          !----------------------------------------------------
          ! Compute the time derivatives from values at space 
          ! time-nodes.
          !----------------------------------------------------
          CALL MATRIX_TIME_GRADIENT (                          &
               & CM ( :, :, :, :, k1-1, l1),                   &
               & DtCM ( :, :, :, :, :, k1-1, l1)  )
          !----------------------------------------------------
          ! Implement the formula (29) of the paper.
          !----------------------------------------------------
          do j = 1, NTime
             do i = 1, NSpace              
                !-------------------------------------------------------
                !-------------------------------------------------------
                !                    
                CM (  :, :, i, j, k1, l1) = DtCM( :, :, i, j, 1, k1-1, l1) 
                !
                do m = l1-1, k1-1

                   if( m > 0) then
                      call MATRIX_D ( m+1, l1, i, j, A, B, DxA, DxB, Dmat)

                      CM (  :, :, i, j, k1, l1)  = CM (  :, :, i, j, k1, l1) + &
                           &   matmul (                    &
                           & CM ( :, :, i, j, k1-1, m), Dmat)
                   end if
                end do
                !
                !----------------------------------------------------
             end do
          end do
          !-------------------------------------------------------------
       end do !  End loop: "do l1 = 1, k1 - 1"
       !----------------------------------------------------------------
    end do ! End loop: "do k1 = 2, Accuracy - 1".

    ! ----------------------------------------------------------------------
  END SUBROUTINE MATRIX_C
  ! ------------------------------------------------------------------------

\end{verbatim}

\subsection{Approximation of space and time derivatives}\label{sec:set-up-nodes-approx}
The interpolation polynomials obtained here can be straightforward generalized to scalar, vector and matrix functions. So, we provide the polynomials for a generic function $f(\xi)$.
To carry out interpolations, we have used the following strategy. 

For interpolation in space, we first cast an interval $ [ x_{i-\frac{1}{2}}, x_{i+\frac{1}{2}}]$ into $ [-\frac{1}{2}, \frac{1}{2}]$ by the change of variable $x = x_{i -\frac{1}{2}} + (\xi - \frac{1}{2}) \Delta x$. Second, we consider $\xi_j = -\frac{1}{2} +  \frac{(j-1)}{ M}$ with $j=1,...,M+1$.  So, the polynomial interpolation in space for $M+1$ order has the form $P(x) = \sum_{k = 0 }^{M} a_{M+1, k } \cdot \xi^{k}$. The coefficients $a_{M+1,k}$ are given in the table \ref{Table:interpolation-space} for $M = 1, 2, 3, 4$, we have used the convention $f_j = f(\xi_j)$. Because of the change of variable, each space derivative of order $l$ needs to be scaled by $ \Delta x^{-l}$. 

To obtain interpolation in time, we cast an interval $[t^n, t^{n+1}]$ into the reference element $ [0 ,1]$ by the change of variable $t = t^n + \tau \Delta t$. Then we construct an interpolation polynomial from the Gaussian points, $\tau_j$ in $[0, 1]$ with $j=1,...,n_{GP}$. Interpolation in time are only needed for accuracy higher than 3. In such a case. The interpolation polynomial has the form $ T(\tau) = \sum_{k=0}^{M-1} b_{M+1, k} \cdot \tau^k$.  The coefficients $b_{M+1,k}$ are shown in the table \ref{Table:interpolation-time} for $M = 2, 3 , 4$. We have used the convention $f_j = f(\tau_j)$. Because of the change of variable, each time derivative of order $l$ needs to be scaled by $ \Delta t^{-l}$.

\begin{table}
\begin{center}
\begin{tabular}{l}
\hline
Second order ($M=1$). \\
\hline
$a_{2,0} = f_2 $ \\
$a_{2,1} = (f_2 - f_1)$ \\

\hline
Third order ($M=2$). \\
\hline
$ a_{3,0} = f_2 $  \\

$ a_{3,0} =(f_3 - f_1) $ \\
 
$ a_{3,0} = 2 ( f_3 - 2  f_2 + f_1 )  $ \\

\hline
Fourth order ($M=3$). \\
\hline
 $ a_{4,0} =  - ( f_4 - 9 f_3 - 9 f_2 + f_1 ) / 16 $ \\ 
 
 $ a_{4,1} =  + ( (- f_4 + 27  f_3 - 27  f_2 + f_1 )) / 8 $ \\

 $ a_{4,2} = ( ( 9  f_4 - 9 f_3 - 9 f_2 + 9 f_1 ) \xi^2 ) / 4 $ \\
 
 $ a_{4,3} = - ( (- 9 f_4 + 27 f_3 - 27 f_2 + 9 f_1 ) \xi^3 ) / 2 $ \\

\hline
Fifth order ($M=4$). \\
\hline

 $ a_{5,0} =   f_3 $ \\ 
 
 $ a_{5,1} =  + ( ( - f_5 + 8  f_4 - 8  f_2 + f_1) ) / 3  $ \\

 $ a_{5,2} = - ( ( 2  f_5 - 32  f_4 + 60  f_3 - 32  f_2 + 2  f_1) ) / 3 $ \\
 
 $ a_{5,3} =  - ( ( - 16  f_5 + 32  f_4 - 32  f_2 + 16  f_1) ) / 3  $ \\

 $ a_{5,4} =   ( ( 32 f_5 - 128  f_4 + 192  f_3 - 128  f_2 + 32  f_1) )/3 $ \\

\hline

\end{tabular}
\end{center}
\caption{Coefficients $a_{M+1,k}$ for interpolation in space, $P(x) = \sum_{k = 0 }^{M} a_{M+1, k } \cdot \xi^{k}$.}\label{Table:interpolation-space}
\end{table}

\begin{table}
\centering
\rotatebox{90}{%

\begin{tabular}{l}

\hline

Third order ($M=2$). \\

\hline
\\

$ b_{3,0} = ( - \sqrt{3} f_2 + f_2 + \frac{ ( \sqrt{3} + 1 )  f_1 ) }{ 2} $ \\
\\
$ b_{3,1} = (  \sqrt{3}  f_2 - \sqrt{3} f_1 )  ) $ \\
\\
\hline

Fourth order ($M=3$). \\

\hline
\\
$ b_{4,0} =  (-\sqrt{15}  f_3 + 5 f_3 - 4 f_2 + (\sqrt{15} + 5) f_1) / 6$ \\
\\
$ b_{4,1} =  -\frac{\left( -\sqrt{15}\, {f_{3}}+10 {f_{3}}-20 {f_{2}}+\left( \sqrt{15}+10\right) \, {f_{1}}\right) }{3} $ \\
\\
$ b_{4,2} =  \frac{\left( 10 {f_{3}}-20 {f_{2}}+10 {f_{1}}\right) }{3} $ \\
\\
\hline

Fifth order ($M=4$). \\

\hline

\\
$ b_{5,0} = - 0.1139171962819898 f_4 + 0.4007615203116506 f_3 - 0.8136324494869276 f_2 + 1.526788125457266 f_1$ \\
\\
$ b_{5,1} = 2.15592710364526 f_4  -7.41707042146264 f_3 + 13.80716692568958 f_2 -8.546023607872199 f_1 $ \\
\\
$ b_{5,2} = - 7.935761849944949 f_4  +24.99812585921913 f_3 - 31.38822236344606 f_2 + 14.32585835417188 f_1 $ \\
\\
$ b_{5,3} = 7.420540068038946 f_4 \cdot - 18.79544940755506 f_3  + 18.79544940755506 f_2 - 7.420540068038946 f_1  
 $ \\
\\
\hline

\end{tabular}
}%
\caption{Coefficients $b_{M+1,k}$ for interpolation in time, $ T(\tau) = \sum_{k=0}^{M-1} b_{M+1, k} \cdot \tau^k$.}\label{Table:interpolation-time}
\end{table}

\subsection{Evaluation of the numerical flux and source term}\label{sec:numerical-flux-source}
Regarding the evaluation of (\ref{eq:flux-source}). The numerical flux can be easily evaluated as 
$ \mathbf{F}_{i+\frac{1}{2} }  = \sum_{j=1}^{n_{GP} } \omega_{j}\mathbf{F}_h( \mathbf{Q}_i(\xi_{M+1}, \tau_j),\mathbf{Q}_{i+1}(\xi_1, \tau_j) )  $, where $\omega_j$ corresponds to the $j$th Gaussian weight in $[0,1]$. 
Whereas, the source terms, can be easily obtained from the interpolation in space and quadrature points in time. Table \ref{Table:interpolation-source} shows the form in which the source is computed for several orders of accuracy. 

\begin{table}
\centering
\rotatebox{90}{%

\begin{tabular}{l}
\hline
Second order ($M=1$). \\
\hline
\\
$ \mathbf{S}_i  = \frac{1}{2} \sum_{j=1}^{n_{GP} } \omega_{j}( 
                                                \mathbf{S}(\mathbf{Q}_i(\xi_2, \tau_j) 
                                              + \mathbf{S}(\mathbf{Q}_i(\xi_1, \tau_j) )  $ \\
                                              
\\
\hline
Third order ($M=2$). \\
\hline
\\ 
$ \mathbf{S}_i  = \frac{1}{6} \sum_{j=1}^{n_{GP}} \omega_{j} ( 
                                                \mathbf{S}(\mathbf{Q}_i(\xi_3, \tau_j) 
                                            + 4 \mathbf{S}(\mathbf{Q}_i(\xi_2, \tau_j) 
                                            +   \mathbf{S}(\mathbf{Q}_i(\xi_1, \tau_j))  $ \\
            
\\
\hline
Fourth order ($M=3$). \\
\hline
\\

$ \mathbf{S}_i  = \frac{1}{8} \sum_{j=1}^{n_{GP}} \omega_{j} ( 
                                                \mathbf{S}(\mathbf{Q}_i(\xi_4, \tau_j) 
                                            + 3 \mathbf{S}(\mathbf{Q}_i(\xi_3, \tau_j) 
                                            + 3 \mathbf{S}(\mathbf{Q}_i(\xi_2, \tau_j))
                                            +   \mathbf{S}(\mathbf{Q}_i(\xi_1, \tau_j))  $ \\

\\ 
\hline
Fifth order ($M=4$). \\
\hline
\\

$ \mathbf{S}_i  = \frac{1}{90} \sum_{j=1}^{n_{GP}} \omega_{j} ( 
                                               7 \mathbf{S}(\mathbf{Q}_i(\xi_5, \tau_j)
                                            + 32 \mathbf{S}(\mathbf{Q}_i(\xi_4, \tau_j) 
                                            + 12 \mathbf{S}(\mathbf{Q}_i(\xi_3, \tau_j) 
                                            + 32 \mathbf{S}(\mathbf{Q}_i(\xi_2, \tau_j))
                                            +  7 \mathbf{S}(\mathbf{Q}_i(\xi_1, \tau_j))  $ \\               
 
\\
\hline

\end{tabular}
}%
\caption{Numerical source term. Here, $\omega_j$ are the corresponding Gaussian weights.}\label{Table:interpolation-source}
\end{table}

\end{document}